\documentclass[12pt,reqno,twoside]{amsart}
\usepackage{amsmath}
\usepackage{amssymb}
\usepackage[all]{xy}
\usepackage{epsfig}
\usepackage{color}
\usepackage{mathabx}
\usepackage{tikz}
\usetikzlibrary{arrows,calc}
\usepackage[english]{babel}
\usepackage{verbatim}
\usepackage{hyperref}

\numberwithin{equation}{section}

\title[The rel leaf of the Arnoux-Yoccoz surface]{The rel leaf and real-rel ray of the Arnoux-Yoccoz surface in genus 3}  
\author{W. Patrick Hooper}
\address{City College of New York, 160 Convent Ave, New York, NY, USA 10463 {\tt whooper@ccny.cuny.edu}}
\author{Barak Weiss}
\address{Tel Aviv University 
{\tt barakw@post.tau.ac.il}}

\font\sn = cmssi8 scaled \magstep0
\font\si = cmti8 scaled \magstep0
\long\def\compat#1{\ifdraft{{\color{blue}\si Pat says ``#1'' }}\else\ignorespaces\fi}
\long\def\combarak#1{\ifdraft{\color{red}\sn Barak says ``#1'' }\else\ignorespaces\fi}

\newcommand\hol{\mathrm{hol}}

\newif\ifdraft\drafttrue
\draftfalse

\newcommand\name[1]{\label{#1}{\ifdraft{\sn [#1]}\else\ignorespaces\fi}}

\newcommand\eq[2]{{\ifdraft{\ \tt [#1]}\else\ignorespaces\fi}\begin{equation}\label{#1}{#2}\end{equation}}
\newcommand {\equ}[1]{\eqref{#1}}

\newcommand{\MM}{{\mathcal{M}}}
\newcommand{\HH}{{\mathcal{H}}}

\newcommand{\Q}{{\mathbb {Q}}}

\newcommand{\R}{{\mathbb{R}}}

\newcommand{\Res}{{\mathrm{Res}}}
\newcommand{\Z}{{\mathbb{Z}}}

\newcommand{\C}{{\mathbb{C}}}

\newcommand{\IE}{{\mathcal{IE}}}

\newcommand{\E}{{\mathbf{e}}}

\newcommand{\A}{{\mathbf{a}}}

\newcommand{\B}{{\mathbf{b}}}

\newcommand{\GL}{\operatorname{GL}}

\newcommand{\Mod}{\operatorname{Mod}}

\newcommand{\SL}{\operatorname{SL}}

\newcommand{\diag}{{\rm diag}}

\newcommand {\ignore}[1]  {}

\newcommand{\spa}{{\rm span}}

\newcommand{\bq}{{\mathbf{q}}}

\newcommand{\LL}{{\mathcal L}}

\newcommand{\FF}{{\mathcal{F}}}

\newcommand{\x}{{\bf x}}

\newcommand{\til}{\widetilde}

\newcommand{\supp}{{\rm supp}}

\newcommand{\rel}{{\mathrm{Rel}}}

\newcommand{\sm}{\smallsetminus}

\newcommand{\vre}{\varepsilon}

\newtheorem{thm}{Theorem}[section]

\newtheorem{lem}[thm]{Lemma}

\newtheorem{prop}[thm]{Proposition}

\newtheorem{cor}[thm]{Corollary}

\newtheorem{remark}[thm]{Remark}

\begin{document}
\date{\today}
\begin{abstract}
We analyze the rel leaf of the Arnoux-Yoccoz translation surface in
genus 3. We show that the leaf is dense in the stratum
$\HH(2,2)^{\mathrm{odd}}$ but that the real-rel trajectory of the
surface is divergent. On this real-rel trajectory, the vertical
foliation of one surface is
invariant under a pseudo-Anosov map (and in particular is uniquely
ergodic), but the vertical foliations on all other surfaces are
completely periodic.
\end{abstract}
\maketitle

\section{Introduction}
A translation surface is a compact oriented surface equipped with a
geometric structure which is Euclidean everywhere except at finitely
many singularities. A stratum is a moduli space of translation
surfaces of the same genus whose singularities share the same combinatorial characteristics. In recent years, intensive study has been
devoted to the study of dynamics of group actions and foliations on
strata of translation surfaces. See \S \ref{sec: basics} for precise
definitions, and see \cite{MT, zorich survey} for surveys.

Let $x$ be a translation surface with $k>1$ singularities. There
is a local deformation of $x$ obtained by moving its singularities
with respect to each other while keeping the holonomies of closed
curves on $x$ fixed. This local deformation gives rise to a foliation
of the stratum $\HH$ containing $x$, with leaves of real dimension $2(k-1)$. In the literature this foliation
has appeared under various names (see \cite[\S 9.6]{zorich
  survey}, \cite{schmoll}, \cite{McMullen-twists} and references therein), and we refer to it as the {\em rel
  foliation}, since nearby surfaces in the same leaf differ only in their relative
periods. A sub-foliation of this foliation, which we will refer to as the
{\em real rel foliation}, is obtained by only
varying the horizontal holonomies of vectors, keeping all vertical
holonomies fixed. Although neither the  rel or the real-rel leaves are given by a group
action, the obstructions to flowing along the real rel foliation are
completely understood (see \cite{MW}) and in the special case $k=2$
it makes sense to discuss the real-rel trajectories $\left\{\rel_r^{(h)}x:
r \in \R \right\}$ of surfaces $x$ without horizontal saddle connections
joining distinct singularities. 

In genus 2, results of McMullen \cite{McMullen-isoperiodic,
  McMullen-cascades} give a detailed understanding of
the closure of rel leaves in the eigenform locus. These results should
be viewed as a
companion to McMullen's classification of closed
sets invariant under the action of $G=\SL_2(\R)$ in genus 2 \cite{McMullen-Annals}. 
A recent breakthrough result of Eskin, Mirzakhani and Mohammadi
\cite{EMM2} has shed light on the $G$-invariant closed
subsets for arbitrary strata, and the purpose of this paper is to
contribute to the 
study of  the topology of closures of rel leaves. Here we focus on the
rel leaf of the Arnoux-Yoccoz surface $x_0$ (see Figure \ref{fig:AY}), which lies in 
the connected component $\HH(2,2)^{\mathrm{odd}}$
of the genus $3$ stratum $\HH(2,2)$. This  
surface, introduced by
Arnoux and Yoccoz in \cite{AY}, is a source of interesting
examples for the theory of translation surfaces. See in particular
\cite{HL}, and the work of Hubert, Lanneau and M\"oller \cite{HLM} in
which the orbit-closure $\overline{Gx_0}$ was determined.  We denote by $\HH$ the sub-locus
of surfaces in $\HH(2,2)^{\mathrm{odd}}$ whose area is the same as
that of $x_0$. 

\begin{figure}
  \hfill
  \begin{minipage}[c]{0.65\textwidth}
   \vspace{0pt}\raggedright
   \includegraphics[scale=0.60]{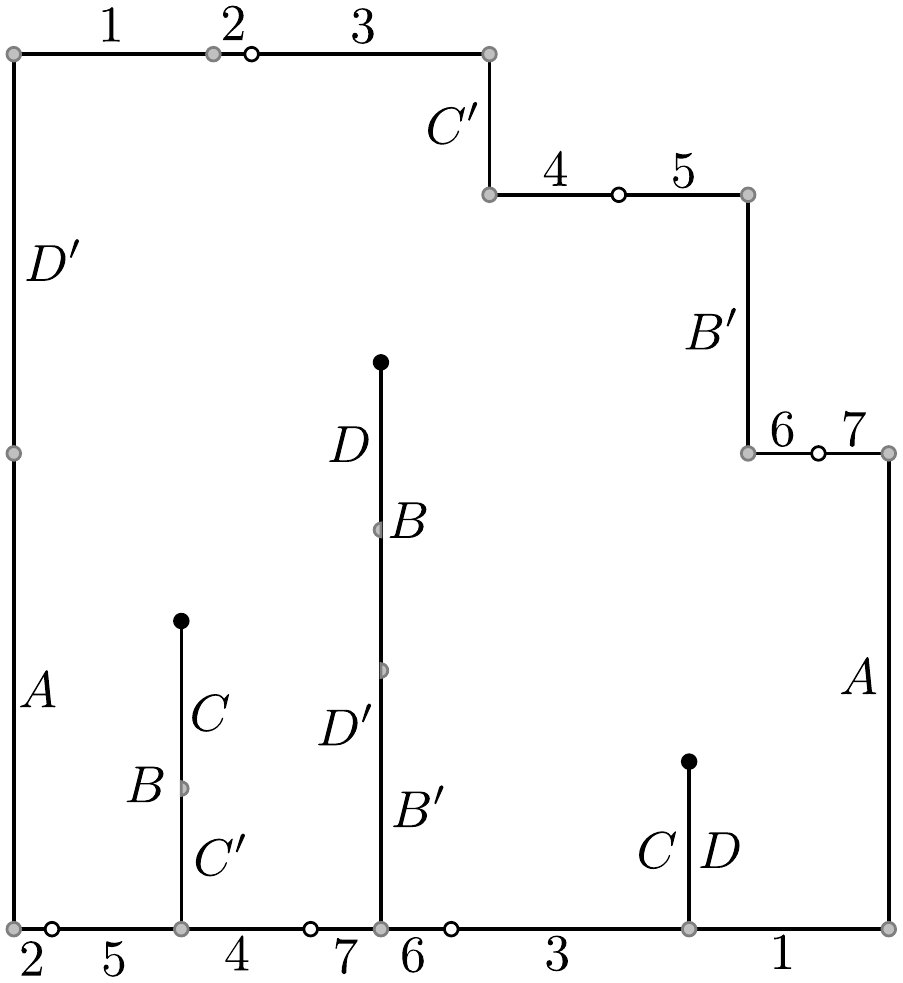}
  \end{minipage}
   \hfill
  \begin{minipage}[c]{0.3\textwidth}
\vspace{0pt}\raggedright
\begin{tabular}{ll}
Label & Edge length \\
\hline
$1$ & $1-\alpha$ \\
$2$ & $2\alpha-1$ \\
$3$ & $\alpha$ \\
$4$, $5$ & $\alpha^2$ \\
$6$, $7$ & $\alpha^3$ \\
$A$ & $2 \alpha$ \\
$B$ &  $\alpha+\alpha^3$ \\
$B'$ & $2 \alpha^2$ \\
$C$, $D$ & $\alpha^2+\alpha^4$ \\
$C'$ & $2 \alpha^3$ \\
$D'$ & $2 \alpha^2+2 \alpha^3$
\end{tabular}  
  \end{minipage}
  \caption{The Arnoux-Yoccoz surface $x_0$  with distinguished
    singularities $x_0$. Edges with the same label are identified 
  by translation, and their lengths are provided by the chart. Black and white points
  denote the two singularities of the surface, and grey dots denote regular points.}
  \label{fig:AY}
  \end{figure}

\medskip 

We now state our results, referring to \S \ref{sec: basics} for
detailed definitions. 

\begin{thm}\name{thm: real rel}
Suppose $x_0$ is as in Figure \ref{fig:AY}, so that the
horizontal and vertical directions are fixed by the pseudo-Anosov map
of \cite{AY}. Then the real rel trajectory of $x_0$ is divergent in $\HH$. Moreover, for any
$r \neq 0$, there is a vertical cylinder decomposition of
$\rel_r^{(h)} x_0$, and the circumferences of the vertical cylinders
tend 
uniformly to zero as $r \to \pm \infty$. 
\end{thm}

There is a symmetry of $x_0$ swapping the vertical and horizontal
directions. (The automorphism $\rho_1$ of \cite{Bowman10} defined under Theorem 2.1 performs this action, though our presentation differs by a diagonal element of $\SL(2,\R)$.) Therefore 
the analogous statement replacing real with imaginary
and vertical with horizontal also holds.

It is well-known that some real-rel trajectories exit their stratum in
finite time due to the collapse of horizontal saddle
connections. In this case the real rel trajectory is not defined for
all time. Theorem \ref{thm: real rel} shows that a real rel trajectory
may be defined for all time and still be divergent in its stratum.

Theorem \ref{thm: real rel} has implications for the study of unique
ergodicity of interval exchange transformations, which we now describe. 
Let $\R^d_+$ denote the vectors in $\R^d$ with positive
entries. 
For a fixed permutation $\sigma$ on $d$ symbols, and 
$\mathbf{a} = (a_1, \ldots, a_d) \in \R^d_+$, let $\mathcal{IE} =
\mathcal{IE}_\sigma(\mathbf{a}): I \to I$ be the interval exchange transformation
obtained by partitioning the interval $I=\left[0, \sum a_i \right)$ into
subintervals of lengths $a_1, \ldots, a_d$ and permuting them
according to $\sigma$. Assuming irreducibility of $\sigma$, it is known that for almost all choices of
$\mathbf{a}$, $\IE_\sigma(\mathbf{a})$ is uniquely ergodic, but
nevertheless that
the set of non-uniquely ergodic interval exchanges is not very
small (see \cite{MT} for definitions, and a survey of this intensively studied
topic). It is natural to expect that all line segments in $\R^d_+$,
other than some obvious 
counterexamples, should inherit the prevalence of uniquely ergodic
interval exchanges. A conjecture in this regard was made by the
second-named author in \cite[Conj. 2.2]{cambridge}, and partial
positive results supporting the conjecture were obtained in 
\cite{MW}. Namely, a special case of \cite[Thm. 6.1]{MW} asserts that 
for any uniquely ergodic $\IE_0 =
  \IE_\sigma(\mathbf{a}_0)$, 
there is an explicitly given hyperplane $\mathcal{L} \subset \R^d_+$
containing $\mathbf{a}_0$, 
such that for any line segment $\ell  = \{\mathbf{a}_s : s \in I\}
\subset \R^d_+$ with $\ell \not \subset \mathcal{L}$, there is an
interval $I_0 \subset I$ containing $0$, such that for almost every $s
\in I_0$, $\IE_\sigma(s)$ is uniquely ergodic. 

\ignore{
\combarak{Instead of stating the result with Yair as you suggested,
  which is technically involved and maybe slows things down, I put in
  an informal discussion in the above paragraph, what do you think?
  Note there is a brief discussion of the paper with Yair in the last
  paragraph of \S \ref{subsection: iets}. There is a version of the
results with Yair (relatively painless special case, but still
slightly intimidating) commented out in
the tex file, if you want to look at it. } \compat{I don't think this
theorem is actually that technical as stated below, and it makes this
work more surprising, so I vote to include it. Maybe less of a
discussion is needed in \S \ref{subsection: iets} if we do this.}

\begin{thm} For any uniquely ergodic $\IE_0 =
  \IE_\sigma(\mathbf{a}_0)$, 
there is an explicit hyperplane $\mathcal{L} \subset \R^d_+$
containing $\mathbf{a}_0$, 
such that for any line segment $\ell  = \{\mathbf{a}_s : s \in I\}
\subset \R^d_+$ with $\ell \not \subset \mathcal{L}$, there is an
interval $I_0 \subset I$ containing $0$, such that for almost every $s
\in I_0$, $\IE_\sigma(s)$ is uniquely ergodic. 
\end{thm}

}
Nevertheless our results provide a  strong
counterexample to \cite[Conj. 2.2]{cambridge}. Recall that a standard
construction of an interval exchange transformation, is to fix a
translation surface $q$ with a segment $\gamma$ transverse to vertical
lines, and define $\IE(q, \gamma)$ to be the first return map to
$\gamma$ along vertical leaves on $q$ (where we parametrize $\gamma$
using the transverse measure $dx$). If $L = \{x(r) : r \in I\}$ is a
sufficiently small 
straight line segment in a stratum of
translations surfaces, $\gamma$ can be chosen uniformly for all $r \in
I$ and the interval exchanges $\IE(x(r), \gamma)$ can be written as
$\IE_{\sigma} (\mathbf{a}(r))$ for some fixed permuation $\sigma$ and some
line $\ell = \{\mathbf{a}(r) : r \in I\} \subset \R^d_+$ (see
\cite{MW} for more details). Thus, taking $x(r) = \rel^{(h)}_r x_0$,
Theorem \ref{thm: real rel} implies:

\begin{cor}\name{cor: Boshernitzan}
There is a uniquely ergodic self-similar interval exchange transformation $\IE_0$
and a line segment $\ell \subset \R^d_+$ such that $\IE_0 =
\IE_\sigma(\mathbf{a}_0)$ with $\mathbf{a}_0$ in the interior of $\ell$,
and such that for all $\mathbf{a} \in \ell \sm \{\mathbf{a}_0\}$,
$\IE_\sigma(\mathbf{a})$ is periodic.  
\end{cor}
In contrast to the real-rel leaf, for the full rel leaf we have:

\begin{thm} \name{thm: rel}
The rel leaf of $x_0$ is dense in 
$\HH$. 
\end{thm}

This is the first stratum in which an
explicit 
dense rel leaf has been described. 
Although we only discuss the rel
leaf of the Arnoux-Yoccoz surface in this paper, it is quite likely
that the rel foliation is ergodic whenever $k>1$, and thus almost all rel
leaves are dense in all relevant strata. Our work establishes
 the existence of dense leaves in $\HH=\HH(2,2)^{\mathrm{odd}}$ but our
arguments can be applied in greater generality. We hope to return
to the general case in future work. Recently 
 Calsamiglia, Deroin and Francaviglia \cite{Deroin} have announced a result
 which  implies ergodicity and also exhibits many explicit dense rel leaves in principal strata,
 i.e. strata all of whose singularities are simple. Their method is
 very different from the one used in this paper. 

We have focused on the Arnoux-Yoccoz surface in genus $3$ but in fact
Arnoux and Yoccoz define one surface in each genus $g \geq 3$,
belonging to $\HH(g-1, g-1)^{\mathrm{odd}}$. (Bowman also gives a treatment of these surfaces \cite{Bowman}.) Many of 
the results of this paper can be generalized to larger $g$. 
Although Theorem \ref{thm: real rel} is not true for $g\geq 4$, the
analogous statement obtained by switching the roles of the vertical and horizontal
directions and imaginary and real rel, is true. (Note that in contrast
to $g=3$, for $g \geq 4$, the Arnoux-Yoccoz surface
  does not seem to have a symmetry swapping the horizontal and vertical
directions.) Also for many $g \geq 4$ (for all such $g \leq 1000$ 
and all such prime $g$), the analogue of Theorem \ref{thm: rel} is true. We
hope to return to this topic elsewhere. 

We now briefly comment on the proofs of our results. 
Since there is a pseudo-Anosov map $\varphi$ fixing $x_0$, there is a
corresponding nontrivial diagonal element $\til g$ with $\til g
x_0=x_0$, and applying this element to  the stratum defines a 
map which acts on the real-rel trajectory $\left\{\rel^{(h)}_rx_0: r
  \in \R
\right\}$ of $x_0$ as an expansion moving points away from 
$x_0$. The proof of Theorem \ref{thm: real rel} relies on an explicit
elementary computation, made
in \S \ref{sec: periods},  which shows that deforming $x_0$ along its
real-rel leaf introduces periodic cylinders, and these cylinders
persist for a full period of the action of $\til g$ on
$\left\{\rel^{(h)}_rx_0: r>0 \right\}$. Our computation does not rely on any prior
theory but to put it  in context, we note
that the surfaces $\rel^{(h)}_rx_0$ have vanishing SAF invariant in
the vertical direction, and the `drift' for the interval exchange
obtained by moving along the vertical direction on $\rel^{(h)}_rx_0$
is sublinear and can be modeled arithmetically in the ring of integers
in a certain cubic field. We refer to \cite{Arnoux88, LPV,
  McMullen-cascades} for discussions of this fascinating
topic. 

In \S \ref{sec: minimal sets} we analyze the interaction of rel
and the horocycle flow on surfaces which have a decomposition into
parallel cylinders, where both rel and the horocycle
flow fix the waist direction of the cylinders. It will be convenient
to choose this direction to be vertical. It turns out that
when there are no vertical saddle connections joining distinct
singularities, both the horocycle and rel flows are given by
linear flows on a torus defined by the twist parameters of the
cylinders in a vertical cylinder decomposition. This parametrization
in terms of twist parameters 
was used for the horocycle flow in \cite{calanque} and is extended
here to rel trajectories. In these coordinates, the horocycle flow orbit-closure is determined
by the moduli of the cylinders while the rel orbit-closure is
determined by their circumferences. This simple observation enables us
to pick up additional invariance in the closure of the rel leaf.
Namely, we find a
family of 3-dimensional tori,
$\{\mathcal{O}_r\}$ defined for all but a discrete set of $r>0$,
each of which is the closure
of the vertical rel leaf of $\rel^{(h)}_rx_0$. 

Our explicit parameterization
enables us to compute the behavior of the tangent planes $T_r$ to
$\mathcal{O}_r$, 
as $r \to 0+$. In \S \ref{sec: endgame} we show that the topological
limit of the $\mathcal{O}_r$ 
contains the entire $V$-orbit of $x_0$, where $V =
\left\{ \left (\begin{matrix} 1 & 0 \\ s & 1 \end{matrix} \right ) : s
\in \R \right \}$ gives the vertical horocycle flow. From this, using \cite{EMM2}, we
deduce that the same topological limit contains the $G$-orbit of $x_0$, and hence, by
\cite{HLM}, also contains the entire hyperelliptic locus $\LL \subset \HH$. 
We then exploit the commutation relations between the rel
foliation and the $G$-action, and ergodicity of the $G$-action, to
conclude the proof of Theorem \ref{thm: rel}.

\subsection{Acknowledgements} This work was stimulated by insightful
comments of Michael Boshernitzan, who conjectured Corollary \ref{cor:
  Boshernitzan}. We are also grateful to
David Aulicino, Josh Bowman, Duc-Manh Nguyen, John Smillie  and Alex Wright for useful
discussions. This collaboration was supported by BSF grant 2010428.
The first author's work is supported by N.S.F. Grant DMS-1101233 and a PSC-CUNY
Award (funded by The Professional Staff Congress and The City University of New
York). The second authors' work was supported by  ERC
started grant DLGAPS 279893.

\section{Basics}\name{sec: basics} 
\subsection{Translation surfaces, strata, $G$-action, cylinders}
\name{sec: translation surfaces}
In this section we define our objects of study and review their basic
properties. We refer to \cite{MT, zorich survey} for more information
on translation surfaces and related notions, and for references for
the statements given in this subsection.

Let $S$ be a compact oriented surface of genus $g \geq 2$, let $\Sigma
= \{\xi_1, \ldots, 
\xi_k\} \subset S$ and let $\mathbf{r} = (r_1,
\ldots, r_k)$ be non-negative integers such that $\sum r_i = 2g-2$. 
A {\em translation atlas} of type $\mathbf{r}$ on $(S, \Sigma)$
is an atlas of charts $(U_{\alpha}, \varphi_{\alpha})$,
where:
\begin{itemize}
\item For each $\alpha$, the set $U_\alpha \subset S\sm \Sigma$ is open, and the map
$$\varphi_\alpha:U_\alpha \to \R^2$$
is continuous and injective. 
\item Whenever the sets $U_\alpha$ and $U_\beta$ intersect, the
  transition functions are local translations, i.e., the maps
%\eq{eq: subs strata}{
$$
\varphi_{\beta} \circ
  \varphi^{-1}_{\alpha}: \varphi_\alpha(U_\alpha \cap U_\beta) \to
\R^2
$$
%} 
are differentiable with derivative equal to the
identity. 
\item around each $\xi_j \in
\Sigma$ The charts glue together to form a cone point with cone angle 
$2\pi(r_j+1)$. 
\end{itemize}
A {\em translation surface structure} on $(S,\Sigma)$ of type $\mathbf{r}$ is an equivalence class of such translation atlases, where $(U_{\alpha}, \varphi_{\alpha})$ and $(U'_{\beta},
\varphi'_{\beta})$ are equivalent if there is an orientation preserving  
homeomorphism $h: S\to S$, fixing all points of $\Sigma$, such that 
$(U_{\alpha}, \varphi_{\alpha})$  
is compatible with  $\left(h(U'_{\beta}), \varphi'_{\beta} \circ h^{-1}\right).$
A {\em marked translation surface structure} is an equivalence class of such atlases
subject to the finer equivalence relation where $(U_{\alpha}, \varphi_{\alpha})$ and $(U'_{\beta},
\varphi'_{\beta})$ are equivalent if $h$ can be taken to be isotopic to the identity via an isotopy fixing $\Sigma$.
Thus, a marked translation surface $\bq$ determines a translation surface
$q$ by {\em forgetting the marking}, and we write $q =\pi(\bq)$ to denote this operation.
Note that our convention is that all singularities are labeled.

Pulling back $dx$ and $dy$ from the coordinate charts we obtain two
well-defined closed 1-forms, which we can integrate along any path
$\gamma$ on $S$. If $\gamma$ is a cycle or has endpoints in $\Sigma$
(a relative cycle), then we define
$${\mathrm x}(\gamma, \bq)=\int_\gamma dx \quad \text{and} \quad {\mathrm y}(\gamma, \bq)=\int_\gamma dy.$$
These integrals only depend on the homology class of $\gamma$ in
$H_1(S, \Sigma)$ and the pair of these integrals is the {\em holonomy}
of $\gamma$, 
\eq{eq: defn hol}{\hol(\gamma, \bq) = \left(\begin{matrix} {\mathrm x}(\gamma, \bq) \\
{\mathrm y}(\gamma, \bq) \end{matrix} \right) \in \R^2.
}
We let $\hol(\bq) = \hol(\cdot, \bq)$ be the corresponding element of $H^1(S, \Sigma;
\R^2)$, with coordinates ${\mathrm x}(\bq)$ and ${\mathrm y}(\bq)$ in $H^1(S, \Sigma; \R)$. 
A {\em saddle connection} for a translation surface $q$ is a straight segment which connects
singularities and does not contain singularities in its interior.

 The set of all (marked) translation surfaces on $(S,
\Sigma)$ of type $\mathbf{r}$ is called the {\em stratum of
(marked) translation surface of type $\mathbf{r}$} and is denoted by
$\HH(\mathbf{r})$ (resp. $\HH_{\mathrm{m}}(\mathbf{r})$). 
%We have suppressed
%the dependence on $\Sigma$ from the notation since for a given type
%$\mathbf{r}$ there is an
%isomorphism between the corresponding set of translation surfaces on
%$(S, \Sigma)$ and on $(S, \Sigma')$ for any other finite subset
%$\Sigma' = \left(\xi'_1, \ldots, \xi'_k\right)$. 
%We denote by $\Sigma_\HH$ the set of singularties of some (any) $q \in
%\HH$. 
The map $\hol: \HH_{\mathrm{m}} (\mathbf{r})\to H^1(S, \Sigma; \R^2)$ just defined gives
local charts for $\HH_{\mathrm{m}} (\mathbf{r})$, endowing it (resp. $\HH(\mathbf{r})$) with the
structure of an affine manifold (resp. orbifold).
%To see how this works, fix a triangulation $\tau$ of $S$ with vertices
%in $\Sigma$. Then $\hol(\bq)$ associates a vector in the plane to each
%oriented edge in $\tau$, and hence associates an oriented Euclidean
%triangle to each oriented triangle of $\tau$. If all the orientations
%are consistent, then a translation structure with the same holonomy as
%$\bq$ can be realized explicitly by gluing the Euclidean triangles to
%each other. Let $\HH_{\mathrm{m}}_\tau$ be the set of all translation
%structures obtained in this way (we say that $\tau$ is {\em realized
%  geometrically} in such a structure). Then the restriction $\hol:
%\HH_{\mathrm{m}}_{\tau} \to H^1(S, \Sigma; 
%\R^2)$ is injective and maps onto an open subset. Conversely every
%$\bq$ admits some geometric triangulation (e.g. a Delaunay
%triangulation as in \cite{MS}) and hence $\HH$ is covered
%by the $\HH_{\tau}$, and so these provide an atlas for a linear
%manifold structure on $\HH_{\mathrm{m}}$. We should remark that a topology on
%$\HH_{\mathrm{m}}$ can be defined independently of this, by considering nearly
%isometric comparison maps between different translation structures,
%and that this topology is the same as that induced by the charts of
%hol.  

Let $\Mod(S, \Sigma)$ denote the mapping class group, i.e. the
orientation preserving homeomorphisms of $S$ fixing $\Sigma$
pointwise, up to an isotopy fixing $\Sigma$. 
The map hol is $\Mod(S, \Sigma)$-equivariant. 
%This means that for
%any $\varphi \in \Mod(S, \Sigma)$, $\hol(\bq \circ \varphi) =
%\varphi^* \hol(\bq)$. 
%, which is nothing more than the linearity of the
%holonomy map with respect to its first argument. 
%
The %One can show that the 
$\Mod(S, \Sigma)$-action on $\HH_{\mathrm{m}}$ 
is properly discontinuous. Thus $\HH (\mathbf{r})= \HH_{\mathrm{m}} (\mathbf{r})/\Mod(S, \Sigma)$ is a
linear orbifold and $\pi: \HH_{\mathrm{m}} (\mathbf{r})\to \HH (\mathbf{r})$ is an orbifold covering
map. 
%Since $\Mod(S, \Sigma)$ contains a finite index torsion-free subgroup
%(see e.g. \cite[Chap. 1]{Ivanov}), there is a finite cover
%$\hat{\HH} \to \HH$ such that $\hat{\HH}$ is a manifold, and 
We have
\eq{eq: dimension of HH}{
\dim \HH (\mathbf{r})= \dim \HH_{\mathrm{m}} (\mathbf{r})= %\dim \HH= 
\dim H^1(S, \Sigma; \R^2)= 2(2g+k-1).
}
%The Poincar\'e-Hopf index theorem implies that 
%\eq{eq: Gauss Bonnet}{\sum r_j = 2g-2.}
%See \cite{EMZ, MT, with Yair}
%for more details.
%We will always assume that $\Sigma \neq \varnothing$ which by \equ{eq:
%Gauss Bonnet} implies that $g \geq 2.$ 

There is an action of $G = \SL_2(\R)$ on $\HH(\mathbf{r})$ and on
$\HH_{\mathrm{m}}(\mathbf{r})$ by post-composition on each 
chart in an atlas. The projection $\pi : \HH_{\mathrm{m}}(\mathbf{r}) \to \HH(\mathbf{r})$ is $G$-equivariant. 
The $G$-action is linear in the homology coordinates, namely, given a
marked translation surface structure $\bq$ and $\gamma \in 
H_1(S, \Sigma)$, and given $g \in G$, we have
\eq{eq: G action}{
\hol(\gamma, g\bq) = g \cdot \hol(\gamma, \bq),
}
where on the right hand side, $g$ acts on $\R^2$ by matrix
multiplication. 

We will write 
\[
u_s 
= \left(\begin{array}{cc} 1 & s \\ 0 & 1 
\end{array}
\right),
 \ \ \ \ \, \ 
g_t = \left(\begin{array}{cc} e^{t} & 0 \\ 0 & e^{-t} 
\end{array}
\right), \ \ \ \ \ 
v_{s}=
\left(\begin{array}{cc}
1 & 0 \\
s & 1 
\end{array}
\right).
%r_{\theta}=
%\left(\begin{array}{cc}
%\cos \theta & -\sin \theta \\
%\sin \theta & \cos \theta 
%\end{array}
%\right).
\]
Also we will denote
$$
U = \{u_s : s\in \R\}, \ \ A = \{g_t: t \in \R \}, 
$$
$$
V = \{v_s: s
\in \R\}, \ \ P=AU =
\left(\begin{matrix} * & * \\ 0 & * \end{matrix} 
\right) \subset G. 
$$

The connected components of strata $\HH(\mathbf{r})$ have been
classified by Kontsevich and Zorich. We will be interested in the
particular connected component $\HH(2,2)^{\mathrm{odd}}$ of
$\HH(2,2)$, since it is the component containing $x_0$. For any $\mathbf{r}$, the area of surfaces in
$\HH(\mathbf{r})$ is preserved by the action of $G$ and we  let $\HH$
be a fixed-area sub-locus of a connected component
of $\HH(\mathbf{r})$. The convention usually adopted in the literature
is to normalize area by setting
$\HH$ to be the locus of area-one surfaces, but it will be more
convenient for us to fix the area equal to some constant, e.g. the area of $x_0$.
There is a globally supported measure on $\HH$ which is defined using 
Lebesgue measure on $H^1(S, \Sigma; \R^2)$ and a `cone construction'. It was shown by Masur that the
$G$-action is ergodic with respect to this measure, and in particular,
almost every $G$-orbit is dense. 

An {\em affine automorphism} of a translation surface $q$ is a
homeomorphism of $q$ which leaves invariant the set of singular points and which
is affine in charts. Some authors require affine automorphisms to
preserve orientation but we will allow orientation reversing affine
automorphisms. The derivative of an affine automorphism is a $2 \times
2$ real matrix of determinant $\pm 1$. If this matrix is hyperbolic (i.e. has distinct real
eigenvalues) then the affine automorphism is called a {\em
  Pseudo-Anosov map}. If the matrix is parabolic (i.e. is nontrivial
and has both eigenvalues equal to 1) then the affine automorphism is
called {\em parabolic}. The group of derivatives of orientation
preserving affine automorphisms of $q$ is called the {\em Veech
  group} of $q$. 

Let $I \subset \R$ be a closed interval with interior, let $c>0$ and
let $\R/c\Z$ be the circle of circumference $c$. 
A {\em cylinder} on a translation surface is a subset homeomorphic to
an annulus
which is the image of $I \times \R /c\Z$ for some $I$ and  $c$ as above, under a map
which is a local Euclidean isometry, and which is maximal in the sense
that the local isometry does not extend to $J \times \R/c\Z$ for an
interval $J$ properly containing $I$. The parameter $c$ is called the
{\em circumference} of the cylinder, and the image of $\{t\} \times
\R/c\Z$ for some $t \in \mathrm{int}( I)$ is called a {\em core
  curve}. In this case the two boundary components of the cylinder are
unions of saddle connections whose holonomies are all parallel to that
of the core curve. If a
translation surface $q$ can be
represented as a union of cylinders, which intersect along their
boundaries, then the directions of the holonomies of the
core curves of the cylinders are all the same, and we say that this
direction is {\em completely periodic} and that  $q$ {\em has
a cylinder decomposition} in that direction. 

\subsection{Rel and Real Rel}
We describe the foliation rel as a foliation on 
$\HH_{\mathrm{m}}(\mathbf{r})$ which descends to a well-defined foliation on
$\HH(\mathbf{r})$. 
%The two summands in the splitting
%\eq{eq: splitting}{ 
%H^1(S, \Sigma; \R^2) \cong H^1(S, \Sigma; \R) \oplus H^1(S, \Sigma;
%\R)}
%induce two foliations on $\HH_{\mathrm{m}}(\mathbf{r})$, which we call the {\em
%real foliation} and {\em imaginary foliation} respectively. Also, considering the
We view our cohomology classes as linear maps from the associated homological spaces.
Observe there is a restriction map 
$$\mathrm{Res}:H^1(S, \Sigma ; \R^2) \to H^1(S; \R^2)$$
which is obtained by mapping a cochain $H_1(S, \Sigma; \R) \to \R^2$
to its restriction to the `absolute periods'
$H_1(S; \R) \subset H_1(S, \Sigma; \R)$. This restriction map is part of the exact sequence in cohomology, 
\eq{eq: defn Res}{
H^0(S;\R^2) \to H^0(\Sigma ; \R^2) \to H^1(S, \Sigma ; \R^2)
\stackrel{\mathrm{Res}}{\to} H^1(S; \R^2) \to \{0\},
}
and we obtain a natural subspace 
$$\mathfrak{R} =\ker \mathrm{Res} \subset H^1(S, \Sigma;
\R^2),$$ consisting of the cohomology classes which
vanish on $H_1(S; \R) \subset
H_1(S, \Sigma; \R).$ 
%Since hol is equivariant with respect to the action of the group $\Mod(S,
%\Sigma)$ on $\HH_{\mathrm{m}}$ and $H^1(S, \Sigma; \R^2)$, and 
Since
%, by naturality
%of the splitting \equ{eq:
%  splitting} and 
the sequence \equ{eq: defn Res} is invariant under homeomorphisms in
$\Mod(S, \Sigma)$, the subspace %s $V_1$, 
$\mathfrak{R}$ %and $W = V_1 \cap \ker \, \Res$ are 
is $\Mod(S, \Sigma)$-invariant. 
Since hol is equivariant with respect to the action of the group $\Mod(S,
\Sigma)$ on $\HH_{\mathrm{m}}(\mathbf{r})$ and $H^1(S, \Sigma;
\R^2)$, 
the foliation of $H^1(S, \Sigma;
\R^2)$ by cosets of the subspace $\mathfrak{R}$ induces by pullback a
foliation of $\HH_{\mathrm{m}}(\mathbf{r})$, and descends to a well-defined
foliation on $\HH(\mathbf{r}) = \HH_{\mathrm{m}}(\mathbf{r})/\Mod(S,
\Sigma)$. The area of a translation surface can 
be computed using the cup product pairing in absolute cohomology and
hence the foliation preserves the area of surfaces, and in particular
we obtain a foliation of a fixed  area  sublocus $\HH$ (see
\cite{BSW} for more details). This foliation is called the {\em rel} 
%and{\em real-rel} 
foliation.  Two nearby translation
surfaces $q$ and $q'$ are in the same plaque if the integrals of the flat structures along all 
closed curves are the same on $q$ and $q'$. 
%, and if the integrals of
%curves joining distinct singularities only differ in their
%horizontal component.  
Intuitively, $q'$ is obtained from $q$ by fixing one singularity as a
reference point and moving the other singularity.
%horizontally.
Recall our convention that singularities are labeled, that is $\Mod(S,
\Sigma)$ does not permute the singular points. Using this one can show
that $\Mod(S, \Sigma)$ acts trivially on $\mathfrak{R} \cong H^0(\Sigma; \R)/H^0(S, \R)$ and hence 
the leaves of the rel foliation
are equipped with a natural translation structure, modeled on $\mathfrak{R}$. 
 The leaves of the rel foliation have (real)
dimension $2(k-1)$ 
%and the leaves of the real-rel foliation have dimension $k-1$ 
(where $k=|\Sigma|$). In this paper we will focus on the case $k=2$, so that
rel leaves are 2-dimensional. We can integrate a cocycle $c  \in
\mathfrak{R}$ on any path joining distinct singularities and the
resulting 
vector in $\R^2$ will be independent of the path, since any two paths
differ by an element of $H_1(S)$. Thus in the case $k=2$ we obtain an  identification of $\mathfrak{R}$ with
$\R^2$ by the map $u \mapsto u(\delta)$ for any path joining the
singularities. Our convention for this identification will be that we
take a path $\delta$  oriented from  $\xi_1$ to $\xi_2$.

The existence of a translation structure on rel leaves implies that
any vector $u \in \mathfrak{R}$ determines
an everywhere-defined vector field on $\HH$. We can apply standard
facts about ordinary differential equations to integrate this vector
field. This gives rise to paths $\psi(t) = \psi_{q,u}(t)$ such that
$\psi(0) = q$ and $\frac{d}{dt} \psi(t) \equiv u.$ We will denote the maximal domain
of definition of $\psi_{q,u}$ by $I_{q,u}$. When  $1 \in
I_{q,u}$ we will say that $\rel^uq$ is defined and write
$\psi_{q,u}(1) = \rel^uq$. Also, in the case $k=2$  we will write 
$$
\rel^{(h)}_rq = \rel^uq \text{ when } u = (r,0),$$
and 
$$\rel^{(v)}_sq = \rel^uq \text{ when } u = (0,s). 
$$
These trajectories are called respectively the {\em real-rel} and {\em
  imaginary-rel} trajectories. 
We will use identical notations for $\bq \in \HH_{\mathrm{m}}(\mathbf{r})$, noting that
since $\pi: \HH_{\mathrm{m}} (\mathbf{r})\to \HH(\mathbf{r})$ is an orbifold covering map, $I_{\bq, u} =
I_{q,u}$ and $\pi(\rel^u \bq) = \rel^u q$. 

Note that the trajectories need not be defined for all time, i.e. $I_{q,u}$
need not coincide with $\R$. For instance this will happen when 
a saddle connection
on $q$ is made to have length zero, i.e. if `singularities
collide'. It was shown in \cite{MW} that this is the only
obstruction to completeness of leaves. Namely, in the case $k=2$, the following holds: 

\begin{prop}\name{thm: real rel main}
Let $\HH$ be a stratum with two singular points, let $\bq \in
\HH_{\mathrm{m}}$, and let $u \in \mathfrak{R}$. Then 
the following are equivalent: 
\begin{itemize}
\item 
$\rel^u\bq $ is defined. 
\item
For all saddle connections $\delta$ on $\bq$, and all $s \in [0,1]$,
$$\hol(\bq, \delta) + s \cdot u( \delta)  \neq 0.$$
\end{itemize}
\end{prop}

\begin{cor}\name{cor: real rel main}
If $q$ has two singular points and no horizontal (resp. vertical) saddle connections joining
distinct singularities, then $\rel^{(h)}_rq$ (resp. $\rel^{(v)}_sq$)
is defined for all $r,s\in \R$. 
\end{cor}

From standard results about ordinary differential equations we have
that 
the map $(q,u) \mapsto \rel^u q$ is continuous on its domain of
definition, and 
$$
\rel^{(h)}_{r_1}(\rel^{(h)}_{r_2} q) = \rel^{(h)}_{r_1+r_2 } (q), \ \
\  \rel^{(v)}_{s_1}(\rel^{(v)}_{s_2} q) = \rel^{(v)}_{s_1+s_2 } (q)
$$
(where defined). 
On the other hand we caution the reader that the rel
plane field need not integrate as a group action, i.e. it is easy to
find examples for which 
$$
\rel^{(h)}_{r} \left(\rel^{(v)}_{s} q \right) \neq \rel^{(v)}_{s}
\left(\rel^{(h)}_r q \right).
$$
We let $G$ act on the stratum $\HH$ in the
usual way and also let $G$ act on $\R^2$ by its standard linear
action. The action of $G$ is equivariant for the map $\hol$ used to
define the translation structure on rel leaves, which
leads to the following result (see \cite{BSW}
for more details):
\begin{prop}\name{prop: rel and G commutation}
Let $x$ be a surface with two singular points and let $u \in
\mathfrak{R} \cong \R^2$. 
If
  $\rel^u(x)$ is defined and $g\in G$ then $\rel^{gu}(gx)$ is
  defined 
and $g(\rel^u(x))=\rel^{gu}(gx)$. In particular, if $q$ has no
horizontal saddle connections joining distinct singularities,  then for all $r,s,t \in \R$, 
\eq{eq: rel commutation}{
g_t \rel^{(h)}_r q = \rel^{(h)}_{e^tr} g_t q. %\ \text{ and } \ u_s
%\rel^{(h)}_r q = \rel^{(h)}_r  u_sq.
}
\end{prop} 

%\subsection{Cylinder decompositions}

\subsection{The Arnoux-Yoccoz surface and its symmetries}
\name{sect:AY}
Let $\alpha$ be the unique real solution to the polynomial equation
\eq{eq: alpha}{
\alpha+\alpha^2+\alpha^3=1.}
This number $\alpha$ is approximately $0.5437$. Its algebraic
conjugates are complex and lie outside the unit circle. Hence, its
multiplicative inverse, $\alpha^{-1}$, is a Pisot number, i.e., its
algebraic conjugates all lie within the unit circle.  

In \cite{AY}, Arnoux and Yoccoz introduced the genus three translation
surface $x_0 \in \HH(2,2)^{\mathrm{odd}}$ which was
depicted in Figure \ref{fig:AY}. 
The surface is built from a $2 \times 2$ square with
three slits and a corner cut out as shown. Edge lengths are elements 
of the ring $\Z(\alpha)$ where $\alpha$ is as in \eqref{eq:
  alpha}. Our presentation  
is the surface of \cite[pp. 496-498]{Arnoux88} scaled by a factor of
two to remove the presence of fractions from edge lengths. It will be
convenient for us to fix this particular scaling in our computations
and thus in the remainder of the paper we let $\HH$ denote the sublocus of
$\HH(2,2)^{\mathrm{odd}}$ consisting of surfaces whose area is the
same as that of $x_0$. 

Inspecting the figure and using the
fact that $\alpha$ is cubic, one finds that the $\Z$-module
$\hol(H_1(S, \Sigma; \Z), x_0)$ has rank 6. 
Arnoux and Yoccoz described a pseudo-Anosov automorphism
$\varphi$ of $x_0$, whose derivative is 
\eq{eq: def til g}{
\til g=\left(\begin{array}{rr}
\alpha^{-1} & 0 \\
0 & \alpha \end{array}\right).
}
The fact that $x_0$ admits such a pseudo-Anosov is somewhat
challenging to see. We refer the reader to \cite[p. 498]{Arnoux88} for
an explanation, see also Remark \ref{remark: pseudoanosov}. 

The surface $x_0$ has two singularities each of cone angle $6\pi$, which we
distinguish as a black singularity and a white singularity; see Figure \ref{fig:AY}.
The pseudo-Anosov $\varphi:x_0 \to x_0$ preserves these two singularities.

Aside from the pseudo-Anosov $\varphi$, the surface $x_0$ admits a few
other symmetries. The surfaces admits
two fixed-point free isometries whose derivatives are given by
reflections of the plane in the $x$- and 
$y$-axes \cite{Bowman10}. (Technically, these are not affine
automorphisms of $x_0$ under our definition, since 
they swap the singularities.) The composition of these maps gives
an affine automorphism of derivative $-I$, which is the hyperelliptic
involution of $x_0$.

Hubert and Lanneau \cite{HL} showed that $x_0$ admits no
parabolic affine automorphisms. This led
to the natural question if the Veech group of $x_0$ is elementary, i.e., just a finite extension
of $\langle \til g \rangle$.  
This question was resolved negatively by Hubert, Laneau and M\"oller,
who proved that there is another pseudo-Anosov automorphism of $x_0$
which does not commute with $\varphi$ \cite[Theorem 1]{HLM}. We will
not have a use for this extra symmetry. 

\ignore{
Let $L \subset \HH$ denote the rel leaf through
$x_0$. Recall that $L$ has the structure of a translation surface, where charts to the plane
are given by considering the change in holonomy of a curve joining  
the black singularity to the white singularity as we move through $L$. 
Affine automorphisms of the unmarked translation surface $x_0$ induced
affine automorphisms of the leaf $L$. If $g \in G$ stabilizes
$x_0$, then as a consequence of Proposition \ref{prop: rel and G
  commutation}, we see that $g$ acts on $L$ with derivative $g$.  
}

\subsection{Results of Hubert-Lanneau-M\"oller}In this subsection we
summarize the results of \cite{HLM}. 

A generic surface in $ \HH(2,2)^{\mathrm{odd}}$ does not have a
hyperelliptic involution, but some surfaces do. Let $\LL
\subset \HH$ denote the subset of surfaces with a hyperelliptic
involution. We will need the following:
\begin{prop}\name{prop: HLM}
The subset $\LL$ is of (real) 
codimension 2 in $\HH$ and is transverse to rel leaves. In particular 
$$
\{\rel^v(z): z \in \LL, \, v\in \mathfrak{R}, \, \rel^v(z) \text{ is
  defined} \}
$$
contains an open subset of $\HH$. 
\end{prop}
\begin{proof}
The fact that $\LL$ is of real codimension 2 follows from the explicit
dimension computations in \cite{HLM}. To see transversality, both
$\mathcal{L}$ and any rel leaf $R$ intersecting $\LL$ are linear in
period coordinates, so it suffices to check that they intersect in a
discrete set of points. The 
hyperelliptic involution fixes the two singular points (see \cite{HLM}) and hence its action
fixes each rel leaf $R$ passing through $\mathcal{L}$, acting on $R$
by an affine automorphism $A$, such that a point on $R$ belongs to
$\mathcal{L}$ precisely when it is a fixed point for $A$. Since
$D(A) = -\mathrm{Id}$ the fixed points for $A$ on $R$ are
isolated. 
\end{proof}

The Arnoux-Yoccoz surface $x_0$ is contained
in $\LL$, and we have: 
%This immediately implies: 
%\begin{prop}\name{prop: codim 2}
%Let $\mathcal{U}$ denote the union of rel leaves in $\HH$ which
%intersect $\LL$. Then $\mathcal{U}$ contains an open subset of $\HH$. 
%\end{prop}

\begin{thm}[{\cite[Theorem 1.3]{HLM}}]\name{thm: HLM} 
$\overline{Gx_0} = \LL$. 
\end{thm}

\subsection{Results of Eskin-Mirzakhani-Mohammadi and some
  consequences}
Recent breakthrough results of Eskin, Mirzakhani and Mohammadi
\cite{EM, EMM2} give a wealth of information about orbit-closures for
the actions of $G$ and $P$ on strata of translation surfaces. The
following summarizes the results which we will need in this paper: 
\begin{thm}\name{prop: EMM}
Let $x$ be a translation surface in a stratum $\HH$. Then 
$$
\overline{Gx} = \overline{Px} = \MM,
$$
where $\MM$ is an immersed submanifold of $\HH$ of even (real) dimension which is cut out by linear equations
with respect to period coordinates, and $\MM$ is the support of a
finite smooth invariant measure $\mu$. Moreover 
\eq{eq: P action}{
\frac{1}{T} \int_0^T \int_0^1 (g_tu_s)_* \delta_x \, ds \, dt
\to_{T \to \infty} \mu,
}
where the convergence is weak-* convergence in
the space of probability measures on $\HH$. 
\end{thm}

The following consequence will be crucial for us, and is of independent interest. 
\begin{prop}\name{prop: V suffices}
Suppose that $x$ is a translation surface and $\{g_t x : t \in
\R\}$ is a periodic trajectory for the geodesic flow. Then  
$\overline{Ux} = \overline{Vx} = \overline{Gx}$. 
\end{prop}

\begin{proof}
We prove for $U$, the proof for $V$ being similar. 
Suppose $g_{p_0}x=x, p_0>0$ where $p_0$ is the period for the closed
geodesic. Let $\delta_x$ denote the Dirac measure on $x$. Let 
$\mu$ be the smooth
$G$-invariant measure  for which $\MM = \overline{Gx} = \supp \, \mu$. By \equ{eq:
  P action},  
\eq{eq: lhs}{
\begin{split}
\mu & = 
\lim_{m \to \infty}
\frac{1}{mp_0} \int_0^{mp_0} \int_0^1 (g_tu_s)_* \delta_x \, ds \, dt \\
 & = \lim_{m \to \infty} \frac{1}{mp_0} \int_0^{p_0} 
 \sum_{i=0}^{m-1} \int_0^1 (g_{ip_0+p}u_s)_* \delta_{x} \, ds \, dp. 
\end{split}
}
For each $p \in [0, p_0)$ we write 
\eq{eq: rhs}{
\nu_{p, m} = \frac{1}{m}\sum_{i=0}^{m-1} 
\int_0^1 (g_{ip_0+p} u_s)_* \delta_x \, ds = \frac{1}{m} \sum_{i=0}^{m-1}
\frac{1}{e^{2(ip_0+p)}} \int_0^{e^{2(ip_0+p)}} (u_s g_{p})_*\delta_x
\, 
ds,}
where we have used the commutation relations $g_\tau u_s =
u_{e^{2\tau}s} g_\tau$ and the fact that $g_{ip_0}x =x $ for each $i$.  
Then the right hand-side of \equ{eq: lhs} is $\lim_{m \to \infty} \frac{1}{p_0} \int_0^{p_0}
\nu_{p, m} dp$. Let $\nu_{0,m}$ be the measure corresponding to
$p =0$, then for any $p$, $\nu_{p, m} = g_{p*}\nu_{0,m}$. Take
a subsequence $\{m_j\}$ along which $\nu_{0,m_j}$ converges to a measure
$\nu$ on $\HH$ (where $\nu(\HH) \leq 1$). Then $\nu_{p, m_j} \to_{j \to \infty} g_{p*}
\nu$. The right hand side of \equ{eq: rhs} shows that $\nu$ is
$U$-invariant and therefore so is each $g_{p*} \nu$, and by \equ{eq:
  lhs} we have $\mu = \frac{1}{p_0} \int_0^{p_0} g_{p*} \nu \, dp$
(and in particular $\nu(\HH)=1$). Since $\mu$ is
$G$-ergodic, by the Mautner property (see e.g. \cite{EW}) it is $U$-ergodic. This implies
that $g_{p*}\nu = \mu$ for almost every $p$, and (since $\mu$ is
$\{g_t\}$-invariant), $\nu=\mu$. These considerations are valid for
every convergent subsequence of the sequence $\nu_{0,m}$ and hence
$\nu_{0,m} \to_{m \to \infty} \mu$. Since $\nu_{0,m}$ is obtained by
averaging over the $U$-orbit of $x$,  the orbit $Ux$ is dense in $\supp \, \mu =
\overline{Gx}$, i.e. $\overline{Ux} = \overline{Gx}$.  
\end{proof} 

Combining Theorem \ref{thm: HLM} and Proposition \ref{prop: V suffices} we obtain:
\begin{cor}\name{cor: corollary V}
With the above notations, we have $\overline{V x_0}=\LL$.
\end{cor}

\section{Periodic vertical directions}
\name{sec: periods} 
In this section, we investigate real rel deformations of the
Arnoux-Yoccoz surface $x_0$. 
Recall that Theorem \ref{thm: real rel} asserts that these surfaces
admit vertical cylinder decompositions. In \S \ref{sect:results}, we
will state more detailed results about these cylinder
decompositions. In particular, Theorem \ref{thm: real rel} follows
directly from Lemma \ref{lem: cylinder geometry} and Remark \ref{rem:
  hyperelliptic}. 
In \S \ref{sect:trip}, we will travel through the real rel leaf, and
obtain explicit descriptions of $x_r=\rel^{(h)}_r x_0$ for $1 \leq r <
\alpha^{-1}$, which constitutes a full period under the action of
$\til g$ on the real rel leaf. Our trip through the rel leaf results
in formal proofs of
the results stated in \S \ref{sect:results}.

\subsection{Results on real rel deformations of the Arnoux-Yoccoz surface}
\label{sect:results}
Recall that $x_0$ admits a pseudo-Anosov $\varphi$ with 
derivative $\til g$ as in \equ{eq: def til g}. The map $\varphi$ would have to preserves the finite
set of horizontal saddle connection, multiplying their lengths by
$\alpha^{-1}$, and thus $x_0$ has no horizontal saddle connections. 
As a consequence of Corollary \ref{cor: real rel main} we see that 
the horizontal rel deformation of $x_0$, 
$$r \mapsto
x_r=\rel^{(h)}_{r} x_0$$ 
is defined for all $r \in \R$. Moreover this real-rel trajectory is not periodic, i.e. there
is no $r>0$ such that $x_r=x_0$; indeed if $x_r = x_0$ and $r>0$ is
the minimal number with this property, we can apply
\equ{eq: rel commutation} to obtain $x_0 = x_{\alpha r}$, contradicting the
minimality of $r$. 

Let $S$ be a genus three surface with two distinguished points, a
black point and a white point, whose union is the set $\Sigma$. We use
this surface to mark our translation surfaces. We take an arbitrary
homeomorphism $S \to x_0$ respecting the colors of distinguished
points. This also
leads to a marking of horizontal rel deformations of $x_0$ via
homotopy within the bundle $\bigcup_{r \in \R} x_r$ of surfaces over
$\R$. 
In particular, we have identified the topological objects associated with our surfaces.

Each surface $x_r$ determines a natural cohomology class $\hol(x_r)
\in H^1(S,\Sigma; \R^2)$ via \equ{eq: defn hol}. 
Because of our conventions, $\hol(x_r)$ varies continuously in $r$. We
take our horizontal rel flow to move the white distinguished point
rightward relative to the black point. This means that if $\gamma$ 
is the homology class of a curve moving from the black point to the
white point, then the holonomies with respect to the different
structures satisfy 
$$\hol(\gamma, x_r)=\hol(\gamma, x_0)+(r,0).$$
\begin{figure}
  \includegraphics[scale=0.60]{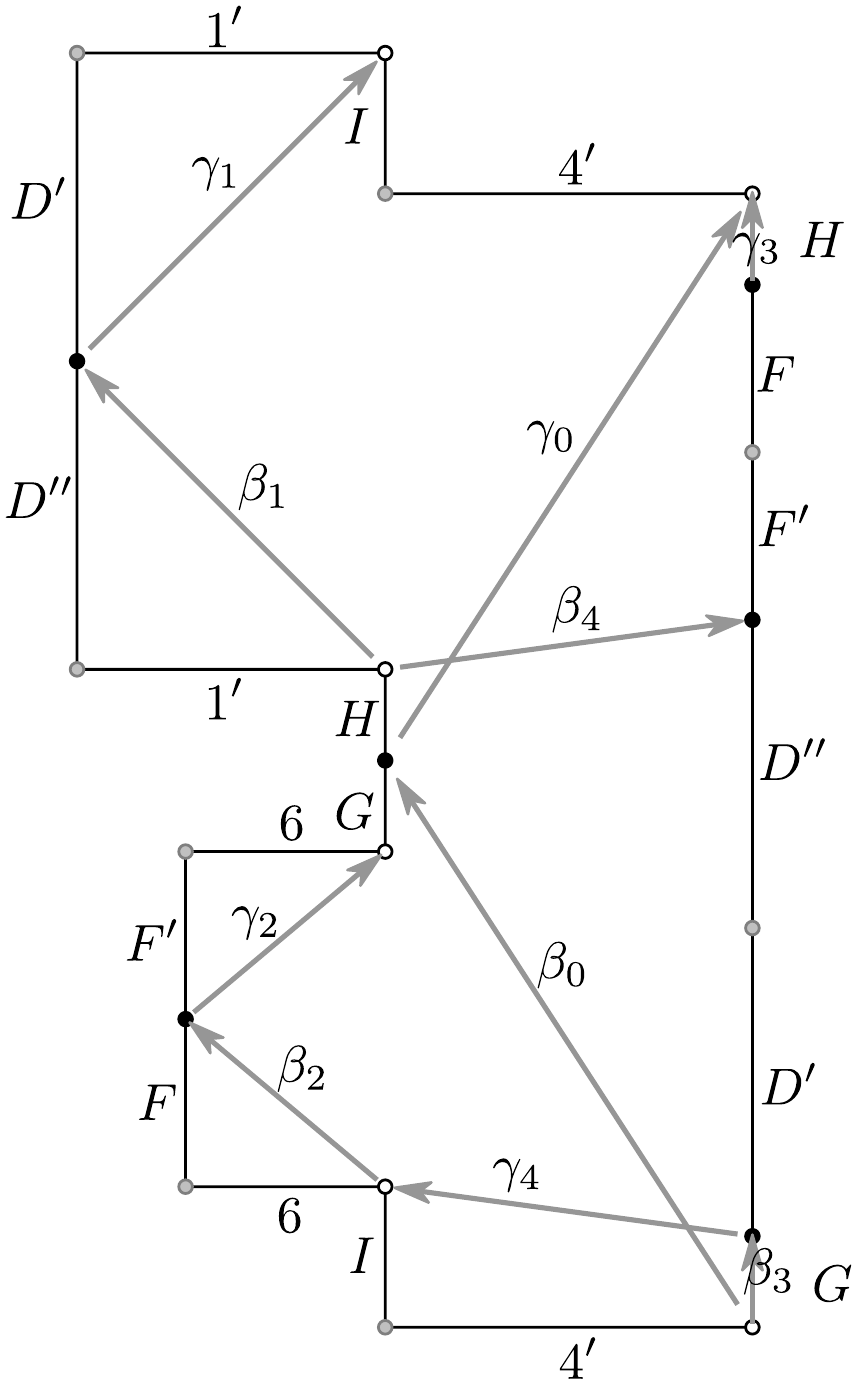}
  \caption{The surface $x_1=\rel^{(h)}_{1} x_0$ with some relative homology classes in $H_1(S,\Sigma; \Z)$.}
  \label{fig:homology}
\end{figure}

We show later in the section that the surface $x_1$ has a presentation
as shown in Figures \ref{fig:homology} and \ref{fig:step3}. Figure
\ref{fig:homology} shows some homology classes
on the surface which will be important to us. Note that the surface
admits a decomposition into three vertical cylinders, and some of the
homology classes are clearly related to these cylinders. These
homology classes belong to a pair of bi-infinite families of homology
classes, 
$\{\beta_k\}, \{\gamma_k\} \subset H_1(S,\Sigma; \Z)$.
Several of these classes are shown in the figure, and we extend inductively according to the 
rules that for all $k \in \Z$, we have:
\begin{equation}
\label{eq:inductive}
\begin{array}{c}
\beta_{k+4}=\gamma_k-\beta_{k+2}-\gamma_{k+2}-2 \gamma_{k+3}, \\
\gamma_{k+4}=\beta_k-\gamma_{k+2}-\beta_{k+2}-2 \beta_{k+3}.
\end{array}
\end{equation}
(The classes shown in Figure \ref{fig:homology} satisfy these identities.)

Note that $H_1(S, \Sigma; \Z)$ is a $\Z$-module isomorphic to $\Z^7$. 
By inspecting the figure and using induction, one can verify the
following:
\begin{prop}%[Homological generators]
\label{prop:Homological generators}
For each $k \in \Z$, the collection of homology classes
$$\{\beta_k, \gamma_k, \beta_{k+1}, \gamma_{k+1}, \beta_{k+2},
\gamma_{k+2}, \beta_{k+3}, \gamma_{k+3}\}$$ generates $H_1(S, \Sigma;
\Z)$, and are related by the identity 
$$\beta_k+\gamma_k= \beta_{k+1}+ \gamma_{k+1}+ \beta_{k+2}+
\gamma_{k+2}+ \beta_{k+3}+ \gamma_{k+3}.$$ 
\end{prop}

In our trip through the rel-leaf, we will use the following result,
whose proof will be given below:
\begin{lem}
\label{lem:holonomies}
The holonomies of the homology classes defined above are:
\eq{eq: holonomies 1}{
%\begin{equation*}
\begin{array}{c}
\hol(\beta_k, x_r)=\left(\begin{array}{r}
\alpha^{3-k}-r \\
\alpha^k+\alpha^{k+2}
\end{array}\right), \\
\hol(\gamma_k, x_r)=\left(\begin{array}{r}
r-\alpha^{3-k} \\
\alpha^k+\alpha^{k+2}
\end{array}\right).
\end{array}
%\end{equation*}
}
\end{lem}

This immediately gives an explicit relationship between these classes
and the pseudo-Anosov $\varphi:x_0 \to x_0$. 
\begin{cor}%[Action on homology]
\label{cor:action}
For each $k$, $\varphi_\ast(\beta_k)=\beta_{k+1}$ and $\varphi_\ast(\gamma_k)=\gamma_{k+1}$.
\end{cor}
\begin{proof}
Observe that an absolute homology class in $H_1(S;\Z)$ is determined
by its holonomy on the surface $x_0$ since both $H_1(S;\Z)$ 
and $\hol(H_1(S ; \Z), x_0)$  are $\Z$-modules of rank $6$. Also recall
that the action of $\varphi$ on $x_0$ 
preserves the two singularities. Fix $k$ and consider the possible
images of $\beta_k$. Observe that each $\beta_k$ is representable as a
path from the white singularity to the black singularity, so its image
is as well. The difference of any two such homology classes is an
absolute class. So, homology classes representable by paths from the
white singularity to the black singularity are also determined by
their holonomy on $x_0$. So, it suffices to observe that  
$$\til g \cdot \hol(\beta_k, x_0)=\hol(\beta_{k+1}, x_0).$$
Similar considerations hold for the classes $\gamma_k$.
\end{proof}

The following lemma describes all vertical cylinders in the surfaces $x_r$ for $r>0$.
\begin{lem}%[Cylinder geometry]
\label{lem: cylinder geometry}
Let $r>0$ and let $k \in \Z$ so that $\alpha^{-k} \leq r <
\alpha^{-(k+1)}$. If $r=\alpha^{-k}$, then $x_r$ admits a
decomposition into three vertical cylinders 
$C_k$, $C_{k+1}$, and $C_{k+2}$ whose core curves represent
the homology classes $\beta_k+\gamma_k$, $\beta_{k+1}+\gamma_{k+1}$
and $\beta_{k+2}+\gamma_{k+2}$, respectively. If $\alpha^{-k} < r <
\alpha^{-(k+1)}$, then $x_r$ admits a decomposition into four vertical 
cylinders $C_k$, \ldots, $C_{k+3}$ whose core curves represent the
homology classes $\beta_k+\gamma_k$, \ldots,
$\beta_{k+3}+\gamma_{k+3}$, respectively.

For each $j \in \Z$, and each $r$ satisfying $\alpha^{-(j-3)}<r<
\alpha^{-(j+1)}$,  the homology
class $\beta_j+\gamma_j$ is represented by a  cylinder $C_j$ in $x_r$. The circumference of this
cylinder is $2 \alpha^j+2 \alpha^{j+2}$, and its variable width is
given by the equation  
$$\mathrm{Width}(C_j, x_r)=\begin{cases}
r-\alpha^{-(j-3)} & \text{for $\alpha^{-(j-3)}<r \leq \alpha^{-j}$}, \\
\alpha^{-(j+1)}-r & \text{for $\alpha^{-j} \leq r < \alpha^{-(j+1)}$}. \\
\end{cases}$$
\end{lem}

\begin{remark}%[Action of the hyperelliptic involution]
\label{rem: hyperelliptic}
A nearly identical statement holds for $x_r$ with $r<0$. Because $x_0$
admits a hyperelliptic involution preserving the singularities, we
have $-I x_r=x_{-r}$ (see \S \ref{sect:AY}). 
\end{remark}

\subsection{A trip through the real rel leaf}
\label{sect:trip}
The goal of this subsection is to find the rel deformations of the
surface $x_r=\rel^{(h)}_{r} x_0$,
with $1 \leq r < \alpha^{-1}$. 
The surface $x_0$ as presented has the white singularity located only
on the top and bottom edges of our cut square. In this case, we can
view action of $\rel^{(h)}_{r}$ as simply changing the length of edges
of the presentation for values of $r$ close to zero. For positive $r$,
we can view it this way until the edge labeled $7$ collapses when
$r=\alpha^3$. At this point the surface becomes as presented in Figure
\ref{fig:step1}. (Checking that edge lengths change as indicated in
the figure follows from the relation 
$\alpha^{k}=\alpha^{k+1}+\alpha^{k+2}+\alpha^{k+3}$. We omit the
straightforward calculations.) 

\begin{figure}
  \begin{minipage}[c]{0.65\textwidth}
   \vspace{0pt}\raggedright
   \includegraphics[scale=0.60]{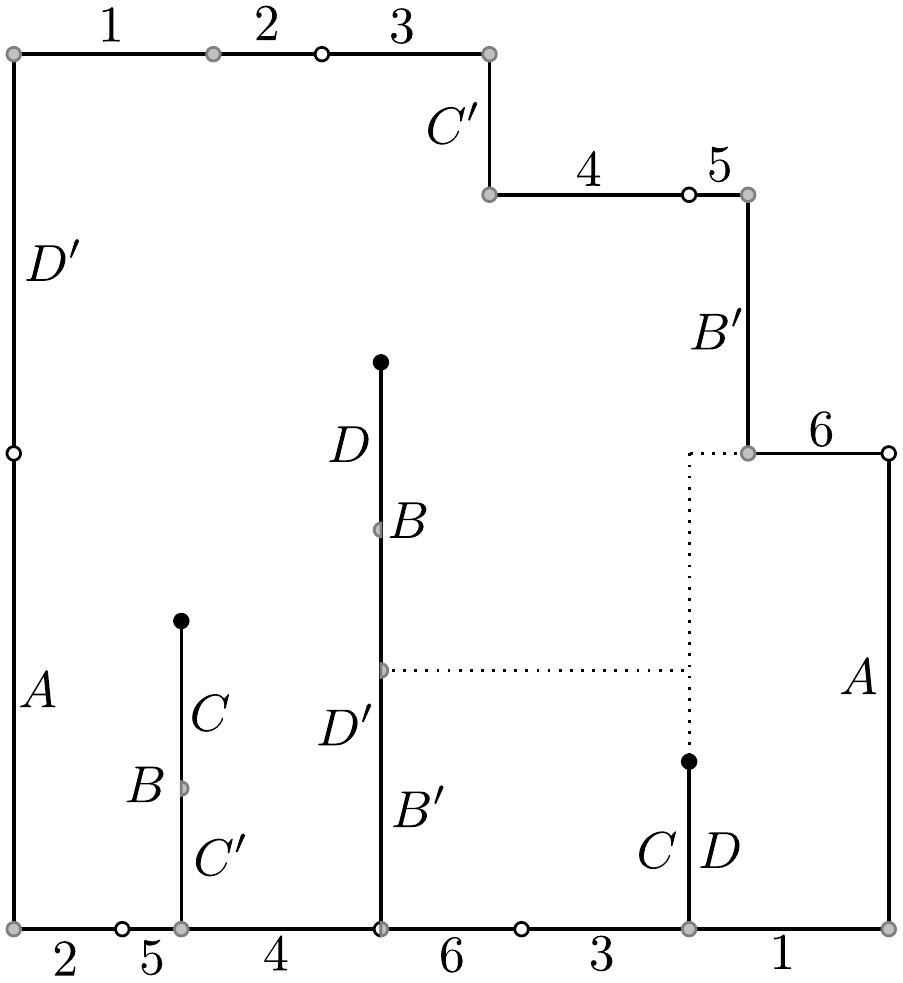}
  \end{minipage}
   \hfill
  \begin{minipage}[c]{0.3\textwidth}
\vspace{0pt}\raggedright
\begin{tabular}{ll}
Label & Edge length \\
\hline
$1$ & $1-\alpha$ \\
$2$ & $\alpha-\alpha^2$ \\
$3$ & $\alpha-\alpha^3$ \\
$4$ & $\alpha^2+\alpha^3$ \\
$5$ & $\alpha^2-\alpha^3$ \\
$6$ & $2 \alpha^3$ \\
$A$ & $2 \alpha$ \\
$B$ &  $\alpha+\alpha^3$ \\
$B'$ & $2 \alpha^2$ \\
$C$, $D$ & $\alpha^2+\alpha^4$ \\
$C'$ & $2 \alpha^3$ \\
$D'$ & $2 \alpha^2+2 \alpha^3$
\end{tabular}
  \end{minipage}
  \caption{The deformed Arnoux-Yoccoz surface  $x_{\alpha^3}=\rel^{(h)}_{\alpha^3} x_0$.}
  \label{fig:step1}
  \end{figure}
  
In order to understand further deformations, we will change the presentation of the surface
$x_{\alpha^3}$. This surface depicted in Figure \ref{fig:step1}
has dotted lines on it. We cut along these dotted lines and reattach
the two rectangles along edges $A$ and $B'$. The new presentation is
shown in Figure \ref{fig:step2}. 
Our changes have the effect of creating some new edges in the figure
(labeled $8$, $E$ and $E'$) and enlarges the edge labeled $6$ into a
new edge we call $6'$. Also, both copies of the edges labeled $1$ and
$2$ are adjacent in the new picture with a regular point between them,
and we define $1'$ to be their union. 

\begin{figure}
  \begin{minipage}[c]{0.60\textwidth}
   \vspace{0pt}\raggedright
   \includegraphics[scale=0.65]{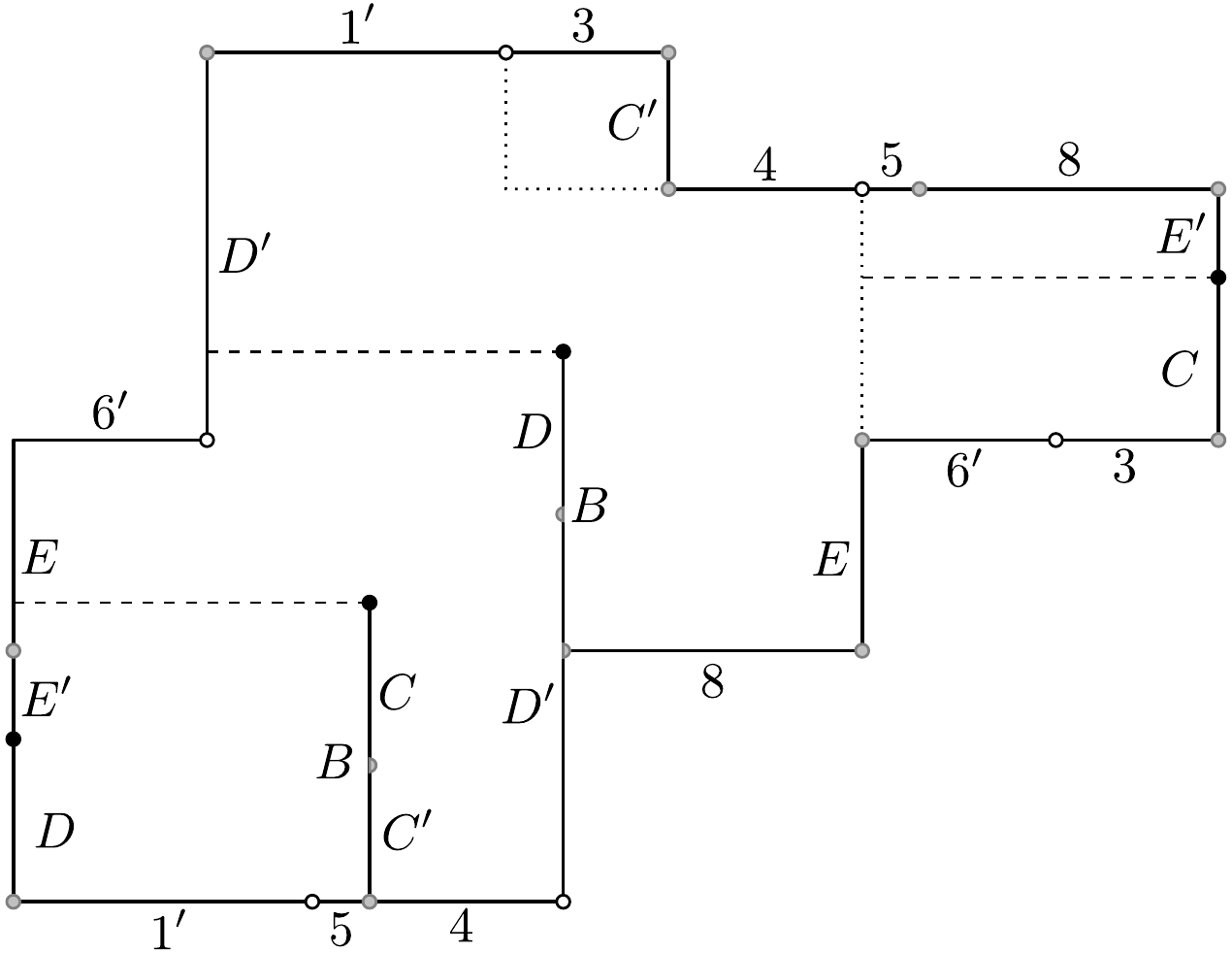}
  \end{minipage}
   \hfill
  \begin{minipage}[c]{0.3\textwidth}
\vspace{0pt}\raggedright
\begin{tabular}{ll}
Label & Edge length \\
\hline
$1'$ & $1-\alpha^2$ \\ 
$3$ & $\alpha-\alpha^3$ \\
$4$ & $\alpha^2+\alpha^3$ \\
$5$ & $\alpha^2-\alpha^3$ \\
$6'$ & $1-\alpha$ \\
$8$ & $\alpha+\alpha^3$ \\
$B$ &  $\alpha+\alpha^3$ \\
$C$, $D$ & $\alpha^2+\alpha^4$ \\
$C'$ & $2 \alpha^3$ \\
$D'$ & $2 \alpha^2+2 \alpha^3$ \\
$E$ & $2 \alpha -2 \alpha^2$ \\
$E'$ & $\alpha^2-\alpha^4$
\end{tabular}
  \end{minipage}
  \caption{The surface $x_{\alpha^3}$, presented differently.}
  \label{fig:step2}
\end{figure}

Figure \ref{fig:step2} presents the surface $x_{\alpha^3}$  in a way which allows us
to understand an additional  rel deformation. We will further
horizontally rel deform by $\alpha+\alpha^2$, which will give us a
total rel deformation by $\alpha+\alpha^2+\alpha^3=1$. We think of
sliding the black singularity left by $\alpha+\alpha^2$ rather than
sliding the white singularity rightward (but this amounts to the same
thing). Note that the black singularity has cone angle $6 \pi$, and so
has three horizontal separatrices leaving leftward. We slit the
surface along these leftward separatrices for length
$\alpha+\alpha^2$. These three segments are shown as dashed lines in
Figure \ref{fig:step2}. We identify edges of the slits so that in the
position where the black singularity was, no singularities
remain. Thus the black point is replaced by three regular points. This
determines the gluings of the slits, and when this is done, the
leftmost points of the three slits are identified to form a new $6
\pi$ cone singularity. This is the image of the black singularity
under the deformation. In order to display the result, we also cut the
surface up a bit more (along the dotted lines) and rearrange the
pieces to show the result. Figure \ref{fig:step3} shows the resulting
surface $x_1$. 
  
\begin{figure}
  \begin{minipage}[c]{0.65\textwidth}
   \vspace{0pt}\raggedright
   \includegraphics[scale=0.60]{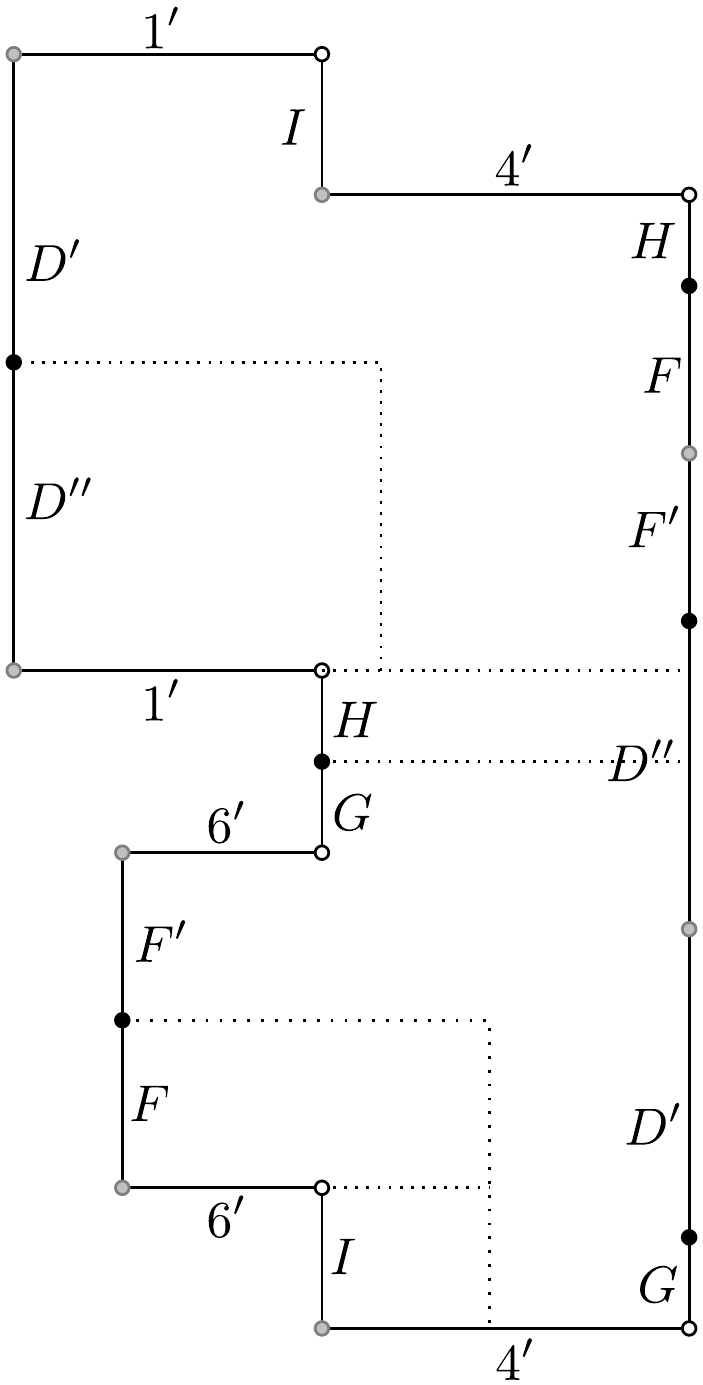}
  \end{minipage}
   \hfill
  \begin{minipage}[c]{0.3\textwidth}
\vspace{0pt}\raggedright
\begin{tabular}{ll}
Label & Edge length \\
\hline
$1'$ & $1-\alpha^2$ \\ 
$4'$ & $\alpha+\alpha^2$ \\
$6'$ & $1-\alpha$ \\
$D'$, $D''$ & $1-\alpha^2$\\
$F$, $F'$ & $\alpha-\alpha^3$ \\
$G$, $H$ & $\alpha^2-\alpha^4$ \\
$I$ & $2 \alpha^3$ \\
\end{tabular}
  \end{minipage}
  \caption{The surface $x_1$ admits a decomposition into three vertical cylinders. Dotted lines show the images of pieces from Figure \ref{fig:step2}.}
  \label{fig:step3}
\end{figure}

We observe that the surface $x_1$ admits a vertical cylinder
decomposition into three cylinders. This is the beginning of the
intervals worth of surfaces $\{x_r~:~1 \leq r < \alpha^{-1}\}$ which
we will analyze more closely. 

Let us pause our trip to note that we have essentially proved the
Lemma which provides the holonomies of our favorite homology classes. 
\begin{proof}[Proof of Lemma \ref{lem:holonomies}]
Figure \ref{fig:homology} shows the homology classes $\beta_k$ and
$\gamma_k$ for $k \in \{0,\ldots, 4\}$.  
The presentation of $x_1$ shown in Figure \ref{fig:homology} is the
same as the one shown in Figure \ref{fig:step3}. 
The latter figure gives the dimensions of the shape, and we can read
off the holonomies of these classes on the surface $x_1$ 
from the figure. We check that they agree with \equ{eq: holonomies 1} when $r=1$. This works when $k \in \{0,\ldots, 4\}$, 
but we can use the inductive formula (\ref{eq:inductive}) to extend
this and observe it holds for all $k \in \Z$ when $r=1$. 
Furthermore because the paths representing the classes $\beta_k$ move from a white
singularity to a black singularity and the paths representing the
classes  $\gamma_k$ do the
opposite, they satisfy: 
$$\hol(\beta_k, x_r)=\hol(\beta_k, x_1)+(1-r,0),$$
$$\hol(\gamma_k, x_r)=\hol(\gamma_k, x_1) +(r-1,0).$$
This observation allows us to verify \equ{eq: holonomies 1} for all $r$.
\end{proof}

We continue the deformation by sliding white singularities to the right relative to the black singularities. As $r$ increases slightly beyond $1$, the surface immediately develops a new vertical cylinder, see Figure \ref{fig:step4}. This figure describes surfaces
$x_r$ for $1 < r < \alpha^{-1}$. 

\begin{figure}
  \begin{minipage}[c]{3in}
   \vspace{0pt}\raggedright
   \includegraphics[scale=0.60]{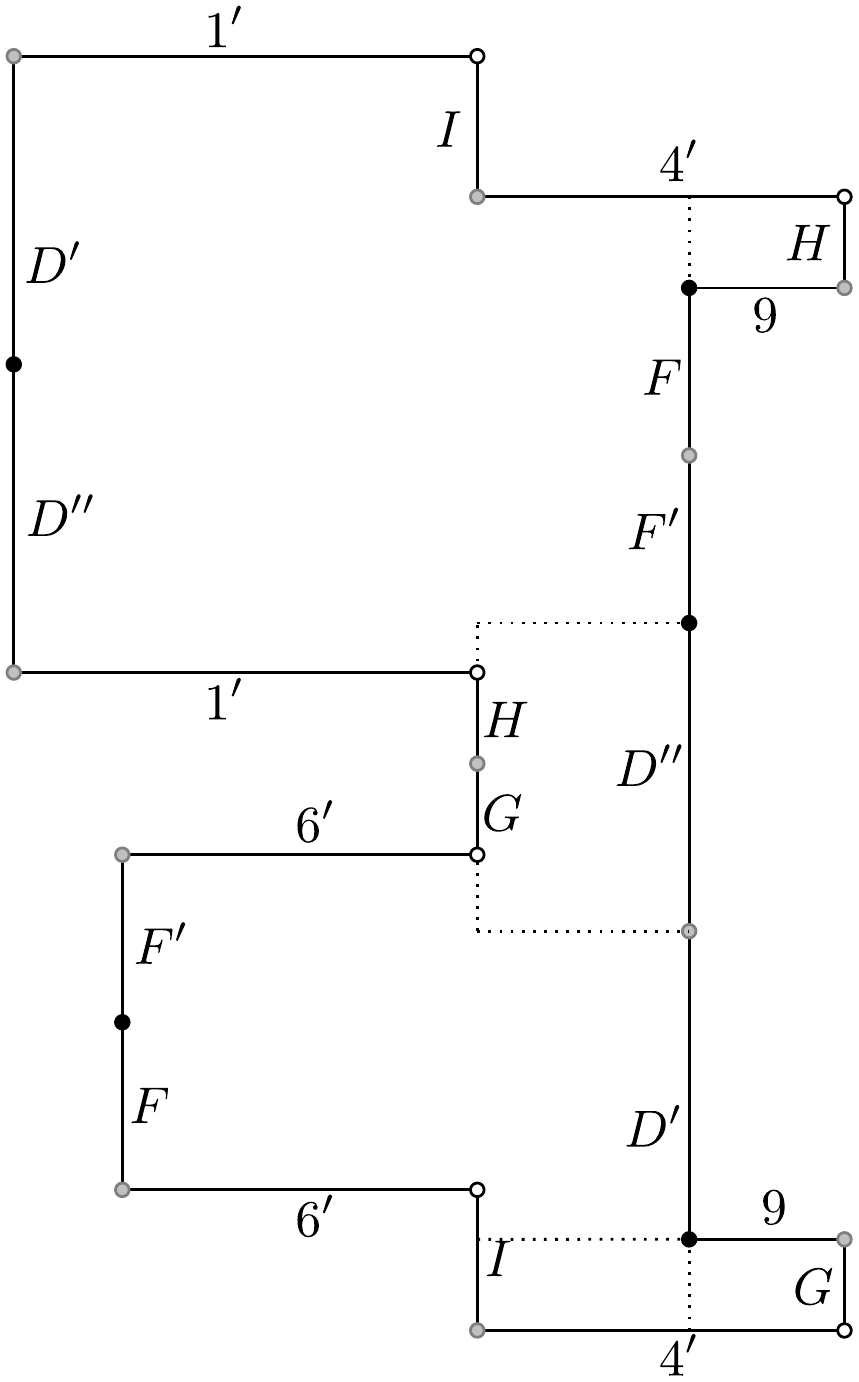}
  \end{minipage}
   \hfill
  \begin{minipage}[c]{0.3\textwidth}
\vspace{0pt}\raggedright
\begin{tabular}{ll}
Label & Edge length \\
\hline
$1'$ & $r-\alpha^2$ \\ 
$4'$ & $\alpha+\alpha^2$ \\
$6'$ & $r-\alpha$ \\
$9$ & $r-1$ \\
$D'$, $D''$ & $1-\alpha^2$\\
$F$, $F'$ & $\alpha-\alpha^3$ \\
$G$, $H$ & $\alpha^2-\alpha^4$ \\
$I$ & $2 \alpha^3$ \\
\end{tabular}
  \end{minipage}
  \caption{The surface $x_r=\rel^{(h)}_{r} x_0$ with $1<r<\alpha^{-1}$   
  admits a decomposition into four vertical cylinders.}
  \label{fig:step4}
\end{figure}

\begin{proof}[Proof of Lemma \ref{lem: cylinder geometry}]
We begin with the first paragraph of the Lemma. Observe that the
statement is true when $r=1$ and when $1<r<\alpha^{-1}$.  
From Figure \ref{fig:homology}, we can see that the upward-oriented core curves of the
vertical cylinders represent the homology classes 
$\beta_0+\gamma_0$, $\beta_1+\gamma_1$ and $\beta_2+\gamma_2$. By
rel-deforming this picture into $x_r$ with $1<r<\alpha^{-1}$, we can
see that when $r$ is in this interval, we get a new cylinder whose core
curve represents $\beta_3+\gamma_3$. Now consider the case of general $r>0$. We
have $\alpha^{-k} \leq r < \alpha^{-(k+1)}$ for some integer $k$ as
stated in the lemma. 
Recall $\til g(x_s)=x_{\alpha^{-1}s}$. Set $s=\alpha^k r$. Then, $1
\leq s < \alpha^{-1}$. Since $\til g^k(x_s)=x_r$, we see that both
$x_s$ and $x_r$ admit a cylinder decompositions into three cylinders
if $r=\alpha^{-k}$ and four cylinders otherwise. Recall that the
action of $\til g$ on homology of these surfaces agrees with the
action of $\varphi$. Using Corollary \ref{cor:action}, we see that the
homology classes of these cylinders are given by  
$$\varphi^k_\ast(\beta_i+\gamma_i)=\beta_{i+k}+\gamma_{i+k}$$
for $i \in \{0,1,2\}$ if $r=\alpha^{-k}$ and for $i \in \{0,1,2,3\}$ otherwise.

As in the lemma, we identify $C_j$ as the cylinder whose core curve
has homology class $\beta_j+\gamma_j$. From the prior paragraph, we
note that this cylinder appears in $x_r$ when $\alpha^{-k} \leq r <
\alpha^{-(k+1)}$ and when $j=i+k$ for $i \in \{0,1,2,3\}$ unless
$r=\alpha^{-k}$ in which case $i=3$ is not allowed. 
Thus, we get a cylinder whose core curve has homology class
$\beta_j+\gamma_j$ when $\alpha^{-(j-3)}<r< \alpha^{-(j+1)}$. The
second paragraph of the Lemma concerns the geometry of these
cylinders. Their circumferences can be determined from the holonomies
of the core curves, which can be computed using Lemma
\ref{lem:holonomies}. 

To compute the widths of the cylinders, observe that for
$1<r<\alpha^{-1}$ the widths of cylinders $C_1$, $C_2$, and $C_3$ are
given by the horizontal holonomies of $\gamma_1$, $\gamma_2$, and
$\gamma_3$, respectively. See Figure \ref{fig:step4}. The width of
$C_0$ on the other other hand is given by the horizontal holonomy of
$\gamma_0-2\gamma_3$. By applying the action of $\til g$ to  these
formulas, and using \equ{eq: rel commutation}, we obtain the widths of all cylinders in
the surface $x_r$ for all $r>0$. 
\end{proof}
\begin{remark}\name{remark: pseudoanosov}
If we were to continue our trip and compute the geometry of
$x_{\alpha^{-1}}$ we would find that $x_{\alpha^{-1}} = \til g
x_1$. This implies that $\til g$ fixes $x_0$, because we would have
$$\til g x_0=\til g \rel^{(h)}_{-1} x_1=\rel^{(h)}_{-\alpha^{-1}} \til g x_1=
\rel^{(h)}_{-\alpha^{-1}} x_{\alpha^{-1}}=x_0,$$
where we use \equ{eq: rel commutation} in the second equality.
This provides an independent proof  that $x_0$ admits a
pseudo-Anosov self-map with derivative $\til g$. 
\end{remark}

\section{Minimal tori for the horocycle flow and vertical
  rel}\name{sec: minimal sets} 
In this section we give basic information about rel leaves of
periodic surfaces. We will work with vertical rel as this is what we
will need in other parts of the paper. 
While the results below are
elementary, they are  of independent interest and we formulate them in
greater generality than we need.

\subsection{Twist coordinates}

Let $\x$ be a marked translation surface with a non-empty labeled singular set $\Sigma$,
and suppose $\x$ is completely periodic in the vertical direction.
Then $\x$ admits a decomposition into vertical cylinders, and
the surface can be recovered by knowing some related combinatorial
data and some geometric parameters. Denote the vertical cylinders by $C_1,
\ldots, C_m$.  
A {\em separatrix diagram } for a vertical cylinder decomposition
consists of the ribbon graph of upward-pointing vertical
saddle connections forming the union of boundaries of cylinders, whose
vertices are given by $\Sigma$ (with orientation at vertices induced
by the translation surface structure on a neighborhood of the singular
point), and an indication of which pairs of
circles in the diagram bound each cylinder $C_i$. For more information
see \cite{KZ}, where this notion was introduced. We add more combinatorial information by selecting 
for each vertical cylinder $C_i$ a rightward-oriented saddle
connection $\sigma_i$ joining the left side of $C_i$ to the right
side. The union of the separatrix diagram and the saddle connections
$\sigma_i$ still has a ribbon graph structure induced by the surface. 
The complete {\em combinatorial data} for our cylinder decomposition
consists of a labeling of cylinders, and the ribbon graph which is the
union of 
the separatrix diagram and the saddle connections $\sigma_i,\ i=1,
\ldots, m$.

Fixing the number and labels of cylinders and the combinatorial data
above, the marked translation surface 
structure on $\x$ is entirely determined by the following {\em parameters}:
the circumferences $c_1, \ldots, c_m \in \R_{>0}$ of the cylinders; 
the lengths $\ell_1, \ldots, \ell_n\in \R_{>0}$
of the vertical saddle connections along the boundaries of the
cylinders; 
and the holonomies 
\eq{eq: if you like it}{
\hol(\sigma_i,\x)=(x_i, y_i) \in \R_{>0} \times \R
}
of each saddle connection $\sigma_i$ for $i=1, \ldots, m$. Observe each $x_i$ records the widths of the
cylinder $C_i$.

We call the numbers $y_i$ the {\em twist parameters}. Observe that we
can vary the twist parameters at will. That is, given $\x$ as above,
there is a map 
\eq{eq:til Phi}{
\til \Phi:\R^m \to \HH_{\mathrm{m}}; \quad (\hat y_1, \ldots, \hat y_m) \mapsto \hat \x,}
where $\hat \x$ is the surface built to have the same parameters as $\x$ except
with twist parameters given by $\hat y_1, \ldots, \hat y_m$,
and where $\HH_{\mathrm{m}}$ is the space of marked translation
surfaces structures modeled on the same surface 
and singularity set as $\x$.
Let $\pi:
\HH_{\mathrm{m}} \to \HH$ be the natural projection. The image of
$\pi \circ \til \Phi$ is the set of all translation surfaces which
have a vertical cylinder decomposition with the same combinatorial
data as $x = \pi(\x)$, and the same parameters describing cylinder
circumferences, lengths of vertical saddle connections, and widths of
cylinders. We refer to this set as the {\em vertical twist space at
  $x$}, and denote it by $\mathcal{VT}_x$. We wish to explicitly parameterize this space. 

Let $y_1, \ldots, y_m$ denote the twist parameters for $\x$, and choose a second set of twist parameters 
$\hat y_1, \ldots, \hat y_m$. Then $\hat \x$ can be obtained from $\x$ by slicing each cylinder $C_i$ 
along a geodesic core curve and regluing so that the right side has moved upward by $\hat y_i-y_i$. 
Thus
for each $\gamma \in H_1(S,\Sigma; \Z)$, 
\begin{equation}
\label{eq:Phi holonomy}
\hol(\gamma,\hat \x)=\hol(\gamma,\x)+\sum_{i=1}^m 
\big(0,
(\hat y_i-y_i) (\gamma \cap C_i)\big)
\end{equation}
where $\cap$ from $\gamma \cap C_i$ denotes the algebraic intersection pairing,
$$\cap:H_1(S,\Sigma; \Z) \times H_1(S \sm \Sigma; \Z) \to \Z,$$ 
taken between
$\gamma$ and a core curve of the cylinder $C_i$.
Writing
$C_i^\ast \in H^1(S,\Sigma; \Z)$ to denote the cohomology class
defined by 
\begin{equation}
\label{eq:cohomology class of cylinder}
C_i^\ast(\gamma)=\gamma \cap C_i,
\end{equation}
we see the following:
\begin{prop}
\label{prop:Phi derivative}
The map $\til \Phi$ is an affine map whose derivative is 
\eq{eq: satisfies}{
D \til \Phi\left(\frac{\partial}{\partial y_i} \right)=(0,C_i^\ast) \in H^1(S,
\Sigma; \R^2), \ \ i=1, \ldots, m.
}
\end{prop}

Recall that $\til \Phi$ defined in (\ref{eq:til Phi}) maps vectors to elements of
$\HH_{\mathrm{m}}$, i.e.  {\em marked} surfaces. Since holonomies of
the saddle connections $\sigma_i$ distinguish surfaces in the image,
$\til \Phi$ is injective. However, the map 
$\pi \circ \til\Phi:\R^m \to \HH$
is certainly not injective.
In the space $\HH$ we consider translation surfaces equivalent if they
differ by the action of an element
of $\Mod(S,\Sigma)$. Let
$\operatorname{M}(\x) \subset \Mod(S,\Sigma)$ denote the subgroup
consisting of equivalence classes of orientation preserving
homeomorphisms of $\x$ such that:
\begin{itemize}
\item Each cylinder $C_i$ is mapped to a cylinder $C_j$ of the same
  circumference and width.
\item Each upward-pointing vertical saddle connection is mapped to an upward-pointing vertical saddle
  connection of the same length, respecting the orientation. 
\end{itemize}
Note that each element
of $\mathrm{M}(\mathbf{x})$ preserves the image of $\til \Phi$, and 
that distinct twist parameters yield the same surface in $\HH$ if and
only if they differ by an element of $\operatorname{M}(\x)$. 
In light of Proposition
\ref{prop:Phi derivative}, $\operatorname{M}(\x)$ pulls back to an
affine action on the twist parameters, and we obtain an affine homeomorphism
\eq{eq:Phi}{\Phi: \R^m / \operatorname{M}(\x) \to \mathcal{VT}_x
  \subset \HH.}

Some elements of the subgroup $\operatorname{M}(\x)$ are clear. Each
Dehn twist $\tau_i \in \Mod(S,\Sigma)$ 
in each vertical cylinder $C_i$ lies in $\operatorname{M}(\x)$. The
action of $\tau_i$ on twist parameters just affects the twist
parameter $y_i$ of $C_i$ and has the effect of adding $c_i$, the
circumference of $C_i$. The {\em multi-twist subgroup} $\operatorname{M}_0(\x) = \langle
\tau_1, \ldots, \tau_m \rangle$ 
of $\operatorname{M}(\x)$ is
isomorphic to $\Z^m$. Moreover, 
$\mathrm{M}(\mathbf{x})$ acts by permutations on the vertical saddle connections, and
$\operatorname{M}_0(\x)$ is the kernel of this permutation action, and
hence is normal and of finite index in
$\operatorname{M}(\x)$. Thus we  have  
the following short exact sequence of groups
\eq{eq:exact}{
\{1\} \to \operatorname{M}_0(\x) \to \operatorname{M}(\x) \to \Delta \to \{1\},}
where $\Delta$ is a subgroup of the group of permutations of the vertical
saddle connections. 
We set $\mathbb{T} = \R^m /\operatorname{M}_0(\x) \cong \prod_{i=1}^m \R / c_i
\Z $ (an $m$-dimensional torus). By normality, the action of $\operatorname{M}(\x)$ on $\R^m$ descends to an action
on ${\mathbb{T}}$ which factors through $\Delta$ via (\ref{eq:exact}). Thus we have the sequence of covers
$$\R^m \to \mathbb{T}  \to  \R^m / \operatorname{M}(\x) \cong
\mathbb{T}/\Delta.$$
We see that $\mathcal{VT}_x$ is isomorphic to the quotient of a torus by
an finite group of linear automorphisms. 
When $\Delta$ is trivial, we actually have $\R^m /
\operatorname{M}(\x) = \mathbb{T}$. Thus the following holds:

\begin{prop}
\name{prop:Phi}
Suppose that the vertical cylinders of $x$ have distinct
circumferences or distinct widths, and that each 
cylinder has a saddle connection on its boundary whose length is distinct from the
lengths of other saddle connections on the boundary. 
Then $\operatorname{M}(\x)= \operatorname{M}_0(\x)$
and therefore $\Phi: \mathbb{T} \to \mathcal{VT}_x$ is an isomorphism of affine manifolds.
\end{prop}
%\begin{proof}
%The assumptions on saddle connections lengths and cylinder widths
%imply that the permutation action in the preceding paragraph is
%trivial, so $\operatorname{M}(\x)= \operatorname{M}_0(\x)$ and $\Delta$ is trivial. 
%\end{proof}

\begin{remark}%[General case]
The above discussion equips the quotient $\mathcal{VT}_x
\cong \mathbb{T}/\Delta$ with the structure of an affine orbifold,
since it is the quotient of $\mathbb{T}$ by the action of a finite
group of affine automorphisms $\Delta$. For an example in which
$\Delta$ is nontrivial and $\mathcal{VT}_x $ is not a torus, let $x$
be the {\em Escher  
staircase surface} obtained by cyclically gluing $2m$ squares (see e.g.
\cite[Figure 3]{LW}). The surface has $m$ parallel cylinders, the
torus $\mathbb{T}$ is isomorphic to $(\R/\Z)^m$, and the
group $\Delta$ contains the group of cyclic permutations of the
coordinates, realized by homeomorphisms of the surface which go up and
down the staircase.
\end{remark}

\begin{remark}
The maps \equ{eq:til Phi} and \equ{eq:Phi} were used in
\cite{calanque} in order to analyze the horocycle flow on completely
periodic surfaces, but the case in which $\Delta$ is
nontrivial was overlooked. We take this opportunity to rectify an
inaccuracy: in \cite[Prop. 4, case (2)]{calanque} instead of a torus we
should have the quotient of the torus by a finite group. 
This does not affect the validity of other statements in
\cite{calanque}. 
\end{remark}

\subsection{Vertical rel flow in twist coordinates}
Now we will specialize to the setting where $\x \in \HH_{\mathrm{m}}$ has two singularities,
which we distinguish as a black singularity and a white singularity. Figure \ref{fig:cylinder data}
shows an example.
We continue to  suppose that the vertical direction is completely periodic, but now we
will also assume that $\x$ admits no vertical saddle connections joining distinct singularities.
That is, each boundary edge of each vertical cylinder
contains only one of the two singularities. 
We will order the cylinders so that $C_1, \ldots, 
C_k$ have the white singularity on their left and the black singularity on their right,
so that $C_{k +1}, \ldots, C_\ell$ have the black singularity on their left and the white singularity on their right, and $C_{\ell+1}, \ldots, C_m$ have the same singularity on both boundary components.

We observe that the vertical rel flow applied to a surface in the
vertical twist space can be viewed as only changing the twist
parameters. Concretely, $\rel_r^{(v)}$ decreases the twist parameters
of cylinders $C_1, \ldots, C_k$ by $r$, increases the twist parameters
of $C_{k+1}, \ldots, C_\ell$ by $r$, and does not change the twist
parameters of $C_{\ell+1}, \ldots, C_m$. Therefore we find:

\begin{prop}
\label{prop:conjugacy}
The vertical twist space $\mathcal{VT}_x$ is invariant under $\rel^{(v)}$. Define 
$$\vec{w} \in \R^m, \ \ w_i=\begin{cases} -1 & \text{if $i \leq k$}\\
1 & \text{if $k < i \leq \ell,$}\\
0 & \text{if $i >\ell$}.
\end{cases}$$
Then the straightline flow 
$$F^r_{\vec{w}} : \vec{y} \mapsto \vec{y}+r \vec{w}$$ 
on $\R^m$ induces a well-defined straightline flow on $\R^m /
\operatorname{M}(\x)$  and $\Phi$ is a topological conjugacy from this
flow to the restriction of $\rel_r^{(v)}$ to $\mathcal{VT}_x$; that is
$\Phi \circ F^r_{\vec{w}} = \rel^{(v)}_r \circ \Phi$ for every $r$. 
\end{prop}

\begin{proof}
In the case $\operatorname{M}(\x) = \operatorname{M}_0(\x)$, the
vertical twist space $\mathcal{VT}_x$ is isomorphic to the torus $\mathbb{T}$ and it is
straightforward to check that the effect of applying $\rel^{(v)}_r$ on
the twist coordinates is exactly $y_i \mapsto y_i + rw_i$,
giving the required conjugacy. In the general case we need to show
that the action of  $F^r_{\vec{w}}$ and of the group $\Delta$ on
$\mathbb{T}$ commute; indeed this will imply both that the action of
$F^r_{\vec{w}}$ on $\mathbb{T} /\Delta \cong \mathcal{VT}_x$
is well-defined, and that 
$\Phi$ intertwines the straightline flow $F^r_{\vec{w}}$ on
$\R^m/\operatorname{M}(\x)$ with the vertical rel flow $\rel^{(v)}_r$ on
$\mathcal{VT}_x$. 
The definition of the $w_i$ implies that $F^r_{\vec{w}}$ and
$\Delta$ commute provided the permutation action of $\operatorname{M}(\x)$ on the vertical 
cylinders preserves each of the three collections of cylinders $\{C_1, \ldots, C_k\}, \{C_{k+1}, \ldots,
C_\ell\}, \{C_{\ell+1}, \ldots, C_m\}$; this in turn follows from our assumption that
singularities are labeled, and the definition of $\operatorname{M}(\x).$
\end{proof}

Given a real vector space $V$, a {\em $\Q$-structure} on $V$ is a choice of
a $\Q$-linear subspace $V_0$ such that $V = V_0 \otimes_{\Q}
\R$ (i.e. there is a basis of $V_0$ as a vector space over $\Q$, which is a
basis of $V$ as a vector space over $\R$). The elements of $V_0$ are
then called {\em rational} points of 
$V$. If $V_1, V_2$ are vector spaces with $\Q$-structures, then a
linear transformation $T:
V_1 \to V_2$ is said to be {\em defined over $\Q$} if it maps rational
points to rational points. 
There is a natural $\Q$-structure on $H^1(S, \Sigma; \R^2)$, namely
$H^1(S, \Sigma; \Q^2)$. Since the action of $\Mod(S,
\Sigma)$ preserves $H_1(S, \Sigma; \Z)$ this induces a well-defined
$\Q$-structure on $\HH$. Moreover since $\HH$ is an affine manifold locally modeled on
$H^1(S, \Sigma ; \R^2)$, the tangent space to $\HH$ at any $x \in \HH$
inherits a $\Q$-structure. With respect to this $\Q$-structure, we obtain:

\begin{prop}\name{prop: real rel closure}
Retaining the notation above, let 
\eq{eq: retaining}{\mathcal{O}(x) = 
\overline{\{\rel^{(v)}_sx : s \in \R \}}\subset \mathcal{VT}_x.
}
Then ${\mathcal
  O}(x)$ is a $d$-dimensional affine sub-orbifold of $\HH$, where $d$ is the
dimension of the $\Q$-vector space 
\eq{eq: dim}{
%S= 
\spa_{\Q} \left\{\frac{-1}{c_1}, \cdots, \frac{-1}{c_k},
  \frac{1}{c_{k+1}}, \ldots, \frac{1}{c_\ell} \right\} \subset \R.} 
%The tangent space to $\mathcal{O}$ at $x$ 
%is a rational subspace of $H^1(S,\Sigma;\R^2)$ of dimension $d$. 
Moreover, for every $x$, the tangent space to $\mathcal{O}(x)$ is a
$\Q$-subspace of $H^1(S, \Sigma; \R^2)$. 

\end{prop}
\begin{proof}
We will work with the
standard $m$-torus ${\mathbb T}^m=\R^m/\Z^m$, and 
define 
\eq{def:Psi}{
\Psi: {\mathbb T}^m \to \mathcal{VT}_x; \quad (t_1, \ldots, t_m) \mapsto
\Phi(c_1 t_1, \ldots, c_m t_m).
}
%The advantage of
%$\Psi$ is that it conjugates the action of $\operatorname{M}_0(\x)$ to
%an action by {\em integer} translations. 
Then the conjugacy from Proposition \ref{prop:conjugacy} leads to a
semi-conjugacy from the straight-line flow on ${\mathbb T}^m$ in
direction  
$$\vec{v}=\left(\frac{-1}{c_1}, \cdots, \frac{-1}{c_k},
  \frac{1}{c_{k+1}}, 
  \cdots, \frac{1}{c_\ell},0, \ldots, 0\right).$$
The orbit closure of the origin of this straight-line flow is a
rational subtorus, and the tangent space to the origin is the smallest
real subspace $V$ of $\R^m$ defined over $\Q$ and containing the
vector $\vec{v}$. Every rational relation among the coordinates of $\vec{v}$
gives a linear equation with $\Q$-coefficients satisfied by $\vec{v}$, and
vice versa; this implies that the dimension of $V$ is the same as the rational
dimension of \equ{eq: dim}. 

Similarly to \equ{def:Psi}, let $\til \Psi: \R^m \to
\HH_{\mathrm{m}}$ be defined by $\til \Psi (t_1, \ldots, t_m) = \til
\Phi(c_1t_1, \ldots, c_mt_m)$. 
Then $\til \Psi$ intertwines
the action of $\Z^m$ on $\R^m$ by translations, with the action of
$\operatorname{M}_0(\x)$ on the image of $\til \Psi$, and we have $\mathcal{O}(x) =
\pi (\til \Psi(V))$.  
Let $\vec{v}_1, \ldots, \vec{v}_d$ be a basis of $V$ contained in $\Z^m$, and let
$\Gamma \subset \Z^m$ be the sub-lattice $ \langle \vec{v}_1, \ldots, \vec{v}_d \rangle.$ 
The subspace $V$ is fixed by the action
$\Gamma \subset \Z^m$ by translations, and $\til \Psi$ intertwines this
translation action with  a translation action of a subgroup of 
$\operatorname{M}_0(\x)$. Thus in order to prove that the tangent space to
$\mathcal{O}(x)$ is defined over $\Q$, it is enough to prove that any
element of $\operatorname{M}_0(\x)$ acts by translation by a rational 
vector. 

As we have seen (see \equ{eq:Phi holonomy}), the action of each $\tau_i$ on $H^1(S,
\Sigma; \R^2)$ is induced by its action on $H_1(S, \Sigma ; \Z)$ via 
$$\gamma \mapsto \gamma + (\gamma \cap C_i) C_i,$$ 
i.e. by
translating by a vector in $H_1(S, \Sigma; \Z)$. Now the assertion
follows from the definition of the $\Q$-structure on
$H^1(S, \Sigma ; \R^2)$.
%Concretely, if $b_1, \ldots, b_{d} \in S$ is a $\Q$-basis for $S$, and
%$q_{i,j}$ are the unique rationals such that
%\eq{eq: sign as}{\frac{\pm 1}{c_j}=\sum_{i=1}^d q_{i,j} b_i \ \ \ \text{(sign as
%  in  \equ{eq: dim})},}
%then the vectors of the form 
%$$\vec{u}_i=(q_{i,1}, \ldots, q_{i,\ell}, 0,\ldots, 0)$$ 
%are a basis for $V$. 
%
%This implies that the rescaled vectors $\vec{u}'_i = \frac{1}{c_i}
%\vec{u}_i$ are also a basis for $V$. Pushing the vectors $\vec{u}'_i$
%forward under $\Psi$ using \eqref{def:Psi}, \eqref{prop:Phi} and \eqref{eq: satisfies}, we obtain the vectors 
%%\eq{eq:eta}{
%$$
%\eta_i=\sum_{j=1}^\ell (0, q_{i,j} C_j^\ast), \ \ \ i=1, \ldots, d,
%$$
%%}
%which are a basis for the tangent space to $\mathcal{O}(x)$ at $x$. Since these vectors are
%rational, we find that the tangent space is a
%rational subspace. 
\end{proof}

\begin{figure}
    \includegraphics[scale=0.65]{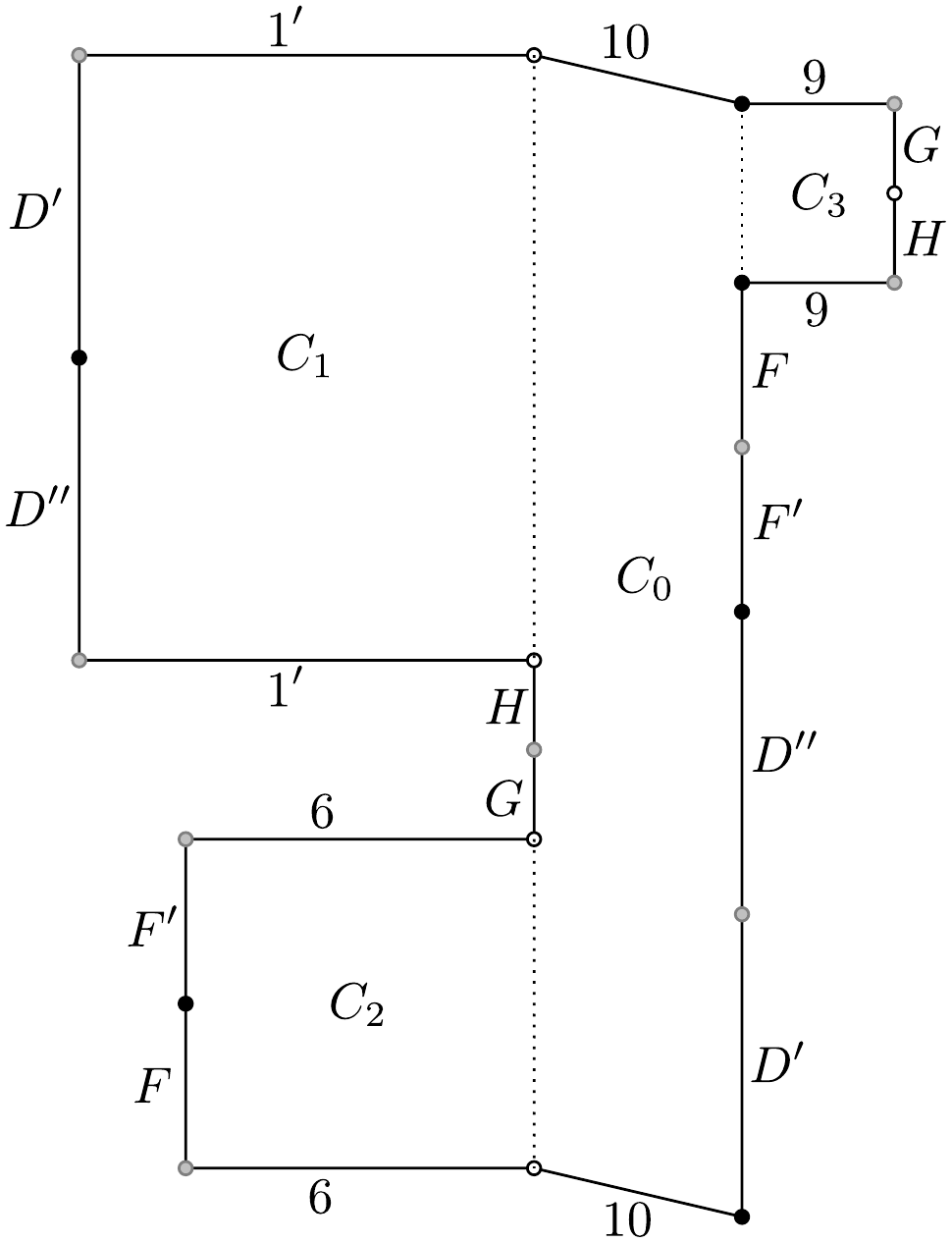}
\caption{Cylinders on the surface $x_r$ with $1<r<\alpha^{-1}$. }
\name{fig:cylinder data}
\end{figure}

\subsection{Vertical rel flow on deformations of the Arnoux-Yoccoz surface}
We now specialize further  to $x_r = \rel^{(h)}_r x_0$. Throughout
this section we assume that $r>0$ and $r$ not an integral power of $\alpha$.
Then $x_r$ admits a decomposition into four cylinders by Lemma \ref{lem: cylinder geometry}.
For such $r$, define $\mathcal{O}_r = \mathcal{O}(x_r)$ via \equ{eq:
  retaining}, and denote the tangent space to $\mathcal{O}_r$ at $x_r$
by $T_r$. Then we have:  
%\eq{eq:torus}{
%{\mathcal O}_r=\overline{\rel^{(v)}_\R x_r}.}
%We will show:
\begin{lem}
\name{lem:tangent space}
With  the notation above,  
$T_r$ is a 3-dimensional $\Q$-subspace of $H^1(S,\Sigma;\R^2)$, 
and ${\mathcal
  O}_r$ 
is a three dimensional affine torus. 
%.
%, and
%the assignment $r \mapsto T_r$ is constant
%on the open interval $(\alpha^{-k},\alpha^{-(k+1)})$ for every $k \in
%\Z$. \combarak{Does this really get used? If so write a proof using
 % commutation relations on the rel leaf.}
\end{lem}
\begin{proof}
Let $k \in \Z$ such that  $\alpha^{-k}<r<\alpha^{-(k+1)}$.
Lemma \ref{lem: cylinder geometry} tells us that $x_r$ has a four cylinder decomposition
with cylinders named $C_k$, $C_{k+1}$, $C_{k+2}$ and $C_{k+3}$, and by
Lemma \ref{lem:holonomies}, their circumferences are 
\eq{eq: their circumferences}{
c_i=2 \alpha^i(1+\alpha^2).
}

Following Proposition \ref{prop: real rel closure} we consider the $\Q$-vector space
$$S=\spa_{\Q} \left\{\frac{-1}{c_k}, \frac{1}{c_{k+1}}, \frac{1}{c_{k+2}}, \frac{1}{c_{k+3}}\right\} \subset \R.$$
(The signs are irrelevant for our purposes, but can be determined by 
noting that there is a $k \in \Z$ so that $\til g^{-k} x_r$ is one of the surfaces
represented by Figure \ref{fig:cylinder data}.) Scaling by $c_{k+3}$
we have $c_{k+3} S = \spa_{\Q} (1, \alpha, \alpha^2, \alpha^3) \cong
\Q(\alpha)$. Since $\alpha$ is cubic, $\dim_{\Q} S = \dim_{\Q} \Q(\alpha)=3$. 
%Proposition \ref{prop: real rel closure}
%shows that the subspace $T_r$ only depends on the homology
%classes and circumferences 
%of the cylinders, and so is independent of $r \in (\alpha^{-k},
%\alpha^{-(k+1)})$.

This proves the first assertion. For the second one, we note by inspecting  Figure \ref{fig:step4} that 
$x_r$ satisfies the conditions of Proposition \ref{prop:Phi} when
$1<r<\alpha^{-1}$. We can extend to all $r$ as above by an
appropriate action of a  power of $\til g$. 
This means
that the map $\Phi$ given in \equ{eq:Phi}
is injective, and the image is a 4-dimensional rational torus in $\HH(2,2)$. By
Proposition \ref{prop: real rel closure}, 
${\mathcal O}_r$ is an affine 3-dimensional sub-torus.  
\end{proof}

%\begin{cor}
%For $r$ as above, ${\mathcal
%  O}_r \subset \HH(2,2)$ 
%is a three dimensional affine torus. 
%\end{cor}
%\begin{proof}
%By inspecting  Figure \ref{fig:step4} we see that 
%$x_r$ satisfies the conditions of Proposition \ref{prop:Phi} when
%$1<r<\alpha^{-1}$. We can extend to all $r$ as stated by an
%appropriate action of a  power of $\til g$. 
%This means
%that the map $\Phi$ given in \equ{eq:Phi}
%is injective, and the image is a 4-dimensional rational torus in $\HH(2,2)$. By
%Proposition \ref{prop: real rel closure}, 
%${\mathcal O}_r$ is an affine 3-dimensional sub-torus.  
%\end{proof}

Define the {\em vertical twist cohomology subspace} $P  $ 
to be the real subspace of $H^1(S,\Sigma;\R^2)$ spanned by the classes of
the form $(0,C^\ast)$, where $C$ varies over the cylinders of
$x_r$. As we have seen, $P$ is the tangent space to
$\mathcal{VT}_r$. 
In light of Proposition \ref{prop:Homological
  generators} and Lemma \ref{lem: cylinder geometry}, $P$ is independent of $r$, and by Proposition \ref{prop: real rel closure}, contains the subspace
$T_r$ for all $r$ as above.

The subspace $P$ is related to horizontal holonomy in
the surface $x_r,$ as follows. If $C_i$ are the vertical cylinders on
$x_r$, $\xi_i$ are their widths (that is $\xi_i$ are the $x_i$ of
\equ{eq: if you like it}), and $C_i^\ast \in H^1(S,\Sigma;\Z)$
are their dual classes as in (\ref{eq:cohomology class of cylinder}), then for
each $\gamma \in H_1(S,\Sigma;\Q)$, 
\eq{eq:holx}{
\hol_{\mathrm x}(\gamma,x_r)=\sum_i \xi_i C_i^\ast(\gamma) %\in
%\Q(\alpha) \oplus r \Q 
.} 
Therefore we see: 
\ignore{
\begin{prop}
\name{prop:bijection}
The subspace $P \subset H^1(S,\Sigma;\R^2)$ is in bijective correspondence with the collection of linear maps
$$L:\Q(\alpha) \oplus r \Q \to \R.$$
Namely, each $w \in H^1(S,\Sigma;\R^2)$ determines a pair of linear maps $w_{\mathrm x}, w_{\mathrm y}:H_1(S,\Sigma;\Q) \to \R$.
We have $w \in P$ if and only if $w_{\mathrm x}=0$ and $w_{\mathrm y}=L \circ \hol_{\mathrm x}$ for some $L$ as above, where 
$\hol_{\mathrm x}:H_1(S,\Sigma;\Q) \to \Q(\alpha) \oplus r \Q$ is defined as in (\ref{eq:holx}).
\end{prop}
\begin{proof}
By (\ref{eq:holx}), we can recover holonomy by evaluating with elements of $P$. Since $P$ is a rational subspace,
it must also contain every cohomology class associated to a pair $(0,L \circ \hol_{\mathrm x})$. 
Let $P'$ denote the collection of cohomology classes associated to such pairs.
We have shown $P' \subset P$. Observe that $P'$ is a vector space of
dimension four because the dimension of the rational vector space
$\Q(\alpha) \oplus r \Q$ is $4$. 
But $P$ is at most $4$ dimensional since the cylinder decompositions
of $x_r$ with $r$ positive and not an integral power of $\alpha$
consist of four cylinders. Since the dimensions match and we have one
inclusion, we know $P=P'$. 
\end{proof}

This point of view can be used to give a description of the action of $\varphi^\ast$ on the subspace $P$:
}
\begin{cor}
\name{cor:action on P}
The action of $\varphi^\ast$ on $H^1(S,\Sigma;\R^2)$ preserves the subspace $P$. 
The cohomology class $\big(0,\hol_{\mathrm x}(x_0)\big) \in
H^1(S,\Sigma;\R^2)$ lies in $P$ and is a dominant eigenvector for the
action of $\varphi^\ast$ 
on $P$. The corresponding eigenvalue is $\alpha^{-1}$. 
\end{cor}
\begin{proof}
Since $\til g x_r = x_{\alpha^{-1}r},$ $\varphi$ maps the cohomology
classes represented by core curves of vertical cylinders on $x_r$, to cohomology
classes represented by core curves of vertical cylinders on
$x_{\alpha^{-1}r}$. Since $P$ is independent of $r$, we find that $\varphi^\ast$
preserves $P$. 
By equation (\ref{eq:holx}) and from the definition of $P$, we see
that $\big(0,\hol_{\mathrm x}(x_r)\big)$ lies in $P$ for all
$r$. 
%Letting $r \to 0$ we find by continuity that
%$\big( 0, \hol_{\mathrm{x}}(x_0) \big ) \in P$. 
By
projecting onto the second summand in $\R^2 = \R \oplus \R$ we
can identify $P$ with a subspace of $H^1(S, \Sigma; \R)$, and we
continue to denote this subspace by $P$, then we have shown that 
$\hol_{\mathrm x}(x_r) \in P$ for all $r$ as above, so letting $r \to
0$ we find that $\hol_{\mathrm{x}}(x_0) \in P$. 
Since $\varphi$ is a pseudo-Anosov on $x_0$ with derivative $\til g$, we know that 
$$\varphi^\ast \hol_{\mathrm x}(x_0)=\alpha^{-1} 
\hol_{\mathrm x}(x_0).$$
Moreover it is known \cite{Fri} that the action of a pseudo-Anosov map on $H^1(S,
\Sigma; \R)$ has a unique dominant eigenvector whose eigenvalue is the
expansion coefficient of the pseudo-Anosov, in this case
$\alpha^{-1}$. Thus $\hol_{\mathrm x}(x_0)$ is an dominant eigenvector
for the action of $\varphi^\ast$ on $H^1(S, \Sigma; \R)$ and hence
also for the action on $P$. 
%\combarak{I don't understand the rest of the proof. It is known that
 % for a pseudo-Anosov map, the eigenvalue of the $2 \times 2$ matrix is dominant for the action
%  on homology. Now you have found an invariant subset of homology and
%  a vector inside which is an eigenvector for this eigenvalue. So
%  doesn't this follow from uniqueness in Perron Frobenius? }
%
%Since $\varphi^\ast$ preserves the set of integer classes
%in cohomology and $P$ is a rational subspace, we know that $P \otimes
%\C$ must also contain two additional complex eigenvectors for
%$\varphi^\ast$ whose eigenvalues are the algebraic conjugates of
%$\alpha^{-1}$. We have a further eigenvector with eigenvalue $1$
%coming from the derivative of the vertical rel deformation at $x_0$
%(which can be seen to lie in $P$ because this derivative at $x_r$ lies
%in $P$ for a sequence of values of $r$ tending to zero 
%and because this derivative varies continuously in $r$). We have
%produced four linearly independent eigenvectors, 
%and $\alpha^{-1}$ is the dominant eigenvalue.
\end{proof}

%Using Proposition \ref{prop:Homological generators}, we observe that
%$T_r \subset H^1(S,\Sigma;\R^2)$ is a subspace of $P$. Since $\varphi$
%preserves $\{x_r\}$ as a set, we know $\varphi^\ast$ stabilizes $P$.  
Let $P_1 = \spa \{ \hol_{\mathrm{x}}(x_0) \} \subset P$. Since $P_1$
is generated by a dominant eigenvector, there is a
$\varphi^\ast$-invariant complementary subspace $P_2$. 
We have:

\begin{cor}
\name{cor:not invariant2}
For any $r$ as above, the subspace $T_r \subset P$ is not contained in $P_2$. 
%In particular, $\spa~ T_r \cup \varphi^\ast(T_r)=P$.
\end{cor}
\begin{proof}
Suppose by way of contradiction that $T_r \subset P_2$. 
By counting dimensions, we obtain that $T_r =
P_2$ and in particular that $\varphi^\ast(T_2) =T_2$. Consider the
$\varphi^\ast$-equivariant projection $\Res: H^1(S,
\Sigma; \R) \to H^1(S ; \R)$ as in \equ{eq: defn Res} (where we take
coefficients in $\R$), and set 
$$
\bar{T}_r = \Res(T_r) \ \ \text{ and } \bar{P} = \Res(P).
$$
By definition of the respective
$\Q$-structures, $\Res$ is defined over
$\Q$, so by Proposition \ref{prop: real rel closure}, $\bar{T}_r$ is a $\Q$-subspace of
$\bar{P}$, which is invariant under $\varphi^\ast$. Since $T_r$
contains the tangent direction to $\rel^{(v)}_r$, which is contained
in the kernel of $\Res$, we find that $\dim \bar{T}_r  = 2$. Thus we
have found a two-dimensional $\varphi^\ast$-invariant $\Q$-subspace of
$\bar{P}$. But one of the eigenvalues of $\varphi^\ast$ on $\bar{P}$ is
the cubic number $\alpha^{-1}$. This is a contradiction. 
\end{proof}
%\begin{cor}
%\name{cor:not invariant}
%The subspace $T_r \subset P$ is not $\varphi^\ast$-invariant. 
%%In particular, $\spa~ T_r \cup \varphi^\ast(T_r)=P$.
%\end{cor}
%\begin{proof}
%Let $T_r^0 \subset P^0 \subset H^1(S; \R^2)$ be the images of $T_r$ and $P$ in absolute cohomology, respectively. Then $\dim P=3$, since the image of rational absolute homology classes under $\hol_{\mathrm x}$ is $\Q(\alpha)$. 
%We know $T_r$ intersects the one dimensional kernel of the projection $P \to P^0$,
%because the derivative of the vertical rel deformation is in $T_r$,
%and this deformation acts trivially on absolute periods. Thus $\dim
%T_r^0=2$.  
%
%Now observe that the only rational subspaces of $P^0$
%which are invariant under $\varphi^\ast$ are the zero subspace and the full space $P^0$. Since
%$T_r^0$ has intermediate dimension, it can not be invariant under $\varphi^\ast$. 
%Therefore, $T_r$ is not $\varphi$-invariant, and $\dim \spa~ T_r \cup \varphi^\ast(T_r) > \dim T_r=3$.
%%Since this span is contained in the four dimensional space $P$, we see this span must be $P$.
%\end{proof}

\ignore{
\begin{thm}\name{thm: crucial}
For any sequence $r_n \to 0$, with $r_n >0$ and not a power of
$\alpha$, we have
%\eq{eq: crucial}{
$$
V x_0 \subset \overline{\,\bigcup_n {\mathcal O}_{r_n}\,}.
$$
%}
\end{thm}
}
\begin{thm}\name{thm: crucial}
For any $r$ as above, let $r_n = \alpha^n r \to 0$. Then 
$$
V x_0 \subset \overline{\,\bigcup_n {\mathcal O}_{r_n}\,}.
$$
\end{thm}

\begin{proof}
In the affine orbifold structure on $\HH$, the orbit $Vx$ is a line,
and the sets $\mathcal{O}_{r_n}$ are linear submanifolds. Since
$x_{r_n} \to x_0$ it suffices to show that the set of accumulation
points of the tangent space $T_{r_n}$ contains the tangent direction
to $V$. By definition of the $V$-action, 
the derivative of the vertical
horocycle flow $\frac{d}{ds}[v_sx_0]$ 
(as an element of the tangent
space to $\HH$ at $x_0$, identified with $H^1(S, \Sigma; \R^2)$) is precisely 
$\big(0,\hol_{\mathrm x}(x_0)\big)$. By Corollary
\ref{cor:action on P},  $\big(0,\hol_{\mathrm x}(x_0)\big)$ is the
dominant eigenvector for the action of $\varphi^\ast$ on $P$. So it is
enough to show that $T_r$ contains a vector which projects
non-trivially onto $P_1$, with respect to the decomposition $P = P_1
\oplus P_2$. But this is immediate from Corollary \ref{cor:not
  invariant2}. 
%need to show that the set of accumulation points of the tangent spaces $T_{r_n}$ contains the tangent direction
%to $V$ contains the top eigenvalue for $\varphi^\ast$. 
%
%First suppose $r \in (1, \alpha^{-1})$, and from Lemma
%\ref{lem:tangent space} that $T_r$ is constant on this interval.  
%We'll refer to this subspace as $T_0$. 
%By Corollary \ref{cor:not invariant}, the subspace $T_0$ is not $\varphi^\ast$ invariant.
%Since $T_r$ is of codimension one in $P$, there must be a class $w
%\in T_0$ which is projectively attracted under iteration of
%$\varphi^\ast$ 
%to the direction of the top eigenvector. That is, there is a sequence $r_j \in \R$ so that
%$\big(\varphi^\ast\big)^j(r_j w)=\big(0,\hol_{\mathrm x}(x_0)\big)$. 
%
%Take a sequence $\{r_n\}$ tending to zero as in the statement of the
%Theorem. Then there is a sequence of integers $j_n$ tending to
%$+\infty$ such that $\alpha^{-j_n} r_n \in (1, \alpha^{-1})$. Then
%$\tilde g^{j_n} x_{r_n}$ 
%is in the family $\{x_r:~ 1<r<\alpha^{-1}\}$ and therefore $T_{r_n}=(\varphi^\ast)^{j_n}(T_r)$. Then
%$$\lim_{n \to \infty} \big(\varphi^\ast\big)^{j_n}(r_{j_n} w)=\big(0,\hol_{\mathrm x}(x_0)\big).$$
%Since each $r_{j_n} w \in T_{r_n}$, we see that 
%$$\big(0,\hol_{\mathrm x}(x_0)\big) \in \overline{\,\bigcup_n T_{r_n}\,}.$$
\end{proof}

\section{The closure of a leaf} \name{sec: endgame}
\begin{proof}[Proof of Theorem \ref{thm: rel}]
Let $\HH$ denote the set of surfaces in
$\HH^{\mathrm{odd}}(2,2)$ with the same area as $x_0$, and let $\Omega \subset \HH$ denote the
closure of the rel leaf of $x_0$. For any $r>0$ 
and $r$ not an integral power of $\alpha$, $\Omega$ contains
$\mathcal{O}_r = \mathcal{O}(x_r)$ (as in\equ{eq: retaining}).
%the closure of
%the vertical rel orbit of $x_r=\rel^{(h)}_r x_0$,
%$${\mathcal O}_r=\overline{\{\rel^{(v)}_s x_r:~ s \in \R\}}.$$
%According to Corollary \ref{cor: apply to AY}, contains the
%torus $T_r$ for each $r$. According to \equ{eq: action on tori}, for
%any such $r$, 
%$\Omega$ contains $\til g^{-k}(T_r)$, and 
Hence by Theorem \ref{thm: crucial}, $\Omega$ contains
the orbit $Vx_0$. 
Now using Corollary \ref{cor: corollary V},
$\Omega$ contains the hyperelliptic locus $\LL$. 

Since $\Omega$ is saturated with respect to the Rel foliation, 
for any $z \in \LL$ and any $u \in \mathfrak{R}$ for which $\rel^u (z)$
is defined, $\Omega$ contains $\rel^u(z)$. Given $z \in \LL$ and $g
\in G$, $gz \in \LL$ since $\LL$ is $G$-invariant. Moreover, if $\rel^u(z)$
is defined, by Proposition \ref{prop: rel and G commutation},
$\rel^{gu}(gz) = g\rel^u(z)$ is also defined and contained in 
$\Omega$. The set $$\{\rel^u(z): z \in \LL, \, u \in
\mathfrak{R},  \rel^u(z) \text{ is defined} \}$$ 
contains an open subset $\mathcal{U}$ of
$\HH(2,2)$ by Proposition \ref{prop: HLM}. Since
$\mathcal{U}$ has positive measure, with respect to the natural flat measure on
$\HH$, we can apply ergodicity of the $G$-action, to find that
$\mathcal{U}$ contains a point $z_0$ for which $\overline{Gz_0} =
\HH$. Since $\overline{Gz_0} \subset \Omega$  we find that $\Omega =
\HH$. 
\end{proof}

\end{document}